\documentclass[12pt,reqno]{NumPDEsArticle}

\usepackage[T1]{fontenc}
\usepackage[english]{babel}

\usepackage{NumPDEsMacros}

\usepackage{csquotes}
\usepackage{enumitem}
\usepackage{amsmath,amssymb}
\usepackage{orcidlink}
\usepackage{diagbox,graphicx}
\usepackage{stmaryrd}

\DeclareMathOperator*{\argmin}{argmin}

\title[AFEM with PGMRES]{Adaptive finite element methods with optimally preconditioned GMRES guarantee 
optimal complexity}

\author{Thomas Führer\orcidlink{0000-0001-5034-6593}}
\address{Facultad de Matem\'{a}ticas, Pontificia 
Universidad Cat\'{o}lica de Chile, Santiago, Chile}
\email{\tt thfuhrer@uc.cl}

\author{Paula Hilbert\orcidlink{0009-0005-0105-1066}}

\author{Ani Miraçi\orcidlink{0000-0003-4962-9662}}
\address{Sorbonne Université, Université Paris Cité, CNRS, Laboratoire Jacques-Louis Lions, LJLL, F-75005 Paris, France}
\email{\tt ani.miraci@sorbonne-universite.fr}

\author{Dirk Praetorius\orcidlink{0000-0002-1977-9830}}
\address{TU Wien, Institute of Analysis and Scientific Computing, Wiedner Hauptstr. 8--10/E101/4, 1040 Vienna, Austria}

\email{\tt paula.hilbert@asc.tuwien.ac.at \quad {\normalfont{(corresponding author)}}}
\email{\tt dirk.praetorius@asc.tuwien.ac.at}


\thanks{This research was funded by the Austrian Science Fund (FWF) projects
\href{https://www.fwf.ac.at/en/research-radar/10.55776/F65}{10.55776/F65} (SFB
F65 ``Taming complexity in PDE systems''),
\href{https://www.fwf.ac.at/en/research-radar/10.55776/I6802}{10.55776/I6802}
(international project ``Functional error estimates for PDEs on unbounded
domains''), 
\href{https://www.fwf.ac.at/en/research-radar/10.55776/PAT3446525}{10.55776/PAT3446525}
(standalone project ``Adaptive Uzawa-type FEM for nonlinear PDEs''), 
and
\href{https://www.fwf.ac.at/en/research-radar/10.55776/PAT3699424}{10.55776/PAT3699424}
(standalone project ``Optimal robust solvers for reliable and efficient AFEMs'')
and by ANID FONDECYT Regular 1250070 project.}

\newcommand{\AS}{\textnormal{AS}}
\newcommand{\SMG}{\textnormal{sMG}}
\newcommand{\Calg}{\const{C}{alg}}
\newcommand{\lalg}{\lambda_\textup{alg}}

\newcommand{\Cbnd}{\const{C}{bnd}}
\renewcommand{\mod}[2]{\textnormal{mod}(#1,#2)}

\begin{document}
\maketitle

\begin{abstract}
We analyze optimal complexity of adaptive finite element methods (AFEMs) for general second-order linear elliptic partial differential equations (PDEs) in the Lax--Milgram setting.
To this end, we formulate an adaptive algorithm which steers the local mesh-refinement as well as the termination of a generalized minimal residual solver (GMRES) with optimal preconditioner to solve the arising non-symmetric finite element systems.
Algorithmic interplay of mesh-refinement and iterative solver is shown to be optimal: A natural and fully computable quasi-error monitoring discretization error and algebraic solver error guarantees unconditional convergence for any choice of adaptivity parameters, i.e., the algorithm cannot fail to converge.
This is ensured algorithmically via a novel adaptive feedback-control for the solver-termination parameter that monitors and ensures full R-linear convergence.
Finally, the quasi-error even decays with optimal rates with respect to the overall computational complexity if the adaptivity parameters are chosen sufficiently small. 
\end{abstract}

\section{Introduction}\label{section:introduction}

The mathematical understanding of adaptive finite element methods (AFEMs) has matured over the last decades.
The seminal works~\cite{doerfler1996, mns2000, bdd2004, stevenson2007, ckns2008} proved optimal convergence rates of AFEM for symmetric second-order linear elliptic partial differential equations (PDEs).
Therein, developments on optimal rates were mainly initiated by corresponding results for adaptive wavelet FEM~\cite{cdd2001, cdd2003}.
This was later extended to general non-symmetric second-order linear elliptic PDEs~\cite{mn2005, cn2012, ffp2014} and put into the abstract framework of the \emph{axioms of adaptivity}~\cite{cfpp2014}, which also covered available AFEM results for nonlinear PDEs like, e.g.,~\cite{gmz2012,bdk2012}.

While the first works~\cite{cdd2001, cdd2003} and~\cite{stevenson2007} also discussed optimal complexity of AFEM, i.e., the optimal decay of the error (resp.\ error estimator) with respect to the overall computational cost, main developments started by~\cite{ckns2008, cfpp2014} were focusing on the optimal decay of the error (resp.\ error estimator) with respect to the number of degrees of freedom.
Beyond 1D, where usual finite element (FE) matrices have a block-tridiagonal structure and can hence be solved at linear cost with respect to the number of degrees of freedom, optimal complexity of AFEM can only be achieved by including an iterative algebraic solver. 
Moreover, optimal complexity indeed requires a thorough interplay of the local mesh refinement with an appropriate iterative solver.
While this is most important for nonlinear PDEs, it should also not be neglected for
linear PDEs. 
Recent works~\cite{ghps2021,bmp2024,bfmps2025} laid the paths of a general theory for such an interplay, where, however, the main focus was on symmetric PDEs (resp.\ nonlinear PDEs, where the linearization leads to symmetric PDE formulations). 

A key ingredient in the analysis of~\cite{ghps2021} is that the iterative solver guarantees uniform contraction with respect to the variational structure of the PDE, i.e., the PDE-related energy norm.
In the case of symmetric linear elliptic PDEs, admissible solvers include preconditioned conjugate gradients (PCG) with optimal multilevel additive Schwarz preconditioners~\cite{cnx2012}, geometric multigrid methods tailored to the multilevel structure of the adaptive mesh hierarchy~\cite{wz2017, imps2022}, or generalized PCG with multigrid preconditioners~\cite{hmp26}. 
As far as non-symmetric second-order linear elliptic PDEs are concerned, the work~\cite{aisfem} formulates an AFEM strategy with optimal complexity for PDEs in the Lax--Milgram setting.
However, adapting ideas from optimal-complexity AFEM for nonlinear PDEs~\cite{hpw2021,hpsv2021}, the iterative solver in~\cite{aisfem} relies on a nested combination of the Zarantonello iteration (to symmetrize the nonsymmetric linear FE systems) with an optimal multigrid method (to solve the symmetrized linear FE system).
Then,~\cite{aisfem} proves optimal complexity provided that the user-chosen parameters are sufficiently small: The damping $\delta>0$ of the Zarantonello iteration, the marking parameter $0 < \theta \le 1$ as well as the parameters $\lambda_{\rm sym}, \lambda_{\rm alg} > 0$ for terminating the Zarantonello iteration and the multigrid solver. However, convergence fails to be unconditional, since only $\theta$ can be chosen arbitrarily. 
In explicit terms, the algorithm from~\cite{aisfem} might fail to converge if $\lambda_{\rm sym}$, $\lambda_{\rm alg}$, or $\delta$ are choosen too large.
Moreover, the standard strategy for solving the nonsymmetric linear FE systems arising for PDEs in the Lax--Milgram setting appears to be GMRES \cite{SaadSchultz86,EES83}. However, without preconditioning, GMRES convergence can be slow or even fail \cite{GPS96_gmres}. To improve convergence, a dedicated preconditioner for the symmetric part of the PDE can be used \cite{Spillane24}, as well as using problem-adapted inner products or weights \cite{PestanaWathen13,GuettelPestana14,SpillaneSzyld24}.

This work focuses on AFEM with GMRES and its unconditional convergence for general second-order linear elliptic PDEs. 
At the solver level, we adopt the left-preconditioning framework of \cite{SarkisSzyld07} and use an optimal symmetric additive Schwarz preconditioner which additionally serves as weight.
From \cite{hmp26}, the resulting weighted preconditioned GMRES method contracts the preconditioned algebraic residual robustly (in the local mesh-size $h$ and in the polynomial degree $p$) in the preconditioner-weighted discrete norm. 
While this discrete contraction is crucial, as is, it is incompatible with the current AFEM analysis~\cite{bfmps2025} as the latter requires contraction of the algebraic error in the PDE-induced energy norm. 
Moreover, to comply with the linear complexity of AFEM, GMRES is restarted after a maximum number of iterations and an a-posteriori error criterion determines when to stop the overall solver and when to proceed with mesh-refinement. 
Overall, the manuscript presents the following novelties: (i) an adaptive algorithm for non-symmetric problems without nested solvers, (ii) a new a-posteriori parameter control criterion included into the adaptive loop,  (iii) a computable quasi-error monitoring the accuracy of the (different sources of) error, (iv) the proof of only finitely many parameter adaptation steps. 
These contributions allow us to prove unconditional full R-linear convergence of the quasi-error independently of the adaptivity parameters provided by the user and optimal convergence rates with respect to total computational cost for sufficiently small adaptivity parameters.

\subsection*{Outline}
The manuscript is organized as follows: 
Section~\ref{section:setting} presents the model problem and its FE discretization.
Section~\ref{section:preconditioned_gmres} recaps the preconditioned GMRES solver in Algorithm~\ref{alg:optimally_preconditioned_gmres} and optimal preconditioners for the principal part from~\cite{hmp26}, and collects potentially well-known properties like equivalence of the discrete GMRES residuals to the solver errors in the functional setting (Lemma~\ref{lem:equivalence_residual_norm}) as well as uniform contraction of the discrete residuals (Proposition~\ref{lem:contraction_gmres}).
Section~\ref{section:afem_algorithm} formulates the new AFEM algorithm with GMRES solver (Algorithm~\ref{algorithm:unconditional_afem_gmres}) with the corresponding quasi-error and a-posteriori error control in Lemma~\ref{lemma:aposteriori}.
Unconditional convergence of Algorithm~\ref{algorithm:unconditional_afem_gmres} is stated and proved in Section~\ref{section:unconditional_full_R_linear_convergence}; see Theorem~\ref{theorem:unconditional_full_R_linear_convergence}. 
Optimal complexity is formulated in Section~\ref{section:adaptive_algorithm}; see Theorem~\ref{th:optimal_complexity}. 
Finally, numerical experiments in Section~\ref{section:numerical_experiments} underline the developed theory.

\color{black}


\section{Setting}\label{section:setting}

\subsection{Model problem}\label{section:model_problem}
On a polyhedral Lipschitz domain \(\Omega \subset \mathbb{R}^d\) 
with \(d \ge 1\)
and given \(f \in L^2(\Omega)\) and 
\(\boldsymbol{f} \in [L^2(\Omega)]^d\),
we consider the diffusion-convection-reaction equation
\begin{equation}\label{eq:model_problem}
	- \operatorname{div} (\boldsymbol{K} \nabla u^\star) 
	+ \boldsymbol{b} \cdot \nabla u^\star + c u^\star 
	= f - \operatorname{div} \boldsymbol{f}
	\quad \text{in } \Omega
	\quad \text{with } u^\star = 0 \quad \text{on } \partial \Omega.
\end{equation}
With the Hilbert space \(\XX \coloneqq H_0^1(\Omega)\) and \(a(v,w) \coloneqq 
\langle \boldsymbol{K} \, \nabla v, \nabla w \rangle_{L^2(\Omega)}\) 
denoting the principal part of the bilinear form 
\(b \colon \XX \times \XX \to \R\) defined by
\begin{equation*}
	b(v, w) \coloneqq a(v, w) 
	+ \langle \boldsymbol{b} \cdot \nabla v, w \rangle_{L^2(\Omega)} 
	+ \langle c \, v, w \rangle_{L^2(\Omega)}
	\quad \text{for all } v, w \in \XX,
\end{equation*}
the weak formulation of \eqref{eq:model_problem} reads as follows: 
Find \(u^\star \in \XX\) such that
\begin{equation}\label{eq:weak_formulation}
	b(u^\star, v) = \langle f, v \rangle_{L^2(\Omega)} 
	+ \langle \boldsymbol{f}, \nabla v \rangle_{L^2(\Omega)} \eqqcolon F(v)
	\quad \text{for all } v \in \XX.
\end{equation}
To ensure well-posedness of this problem, we assume that $\vec{b} \in [L^{\infty}(\Omega)]^d$ and $c \in L^{\infty}(\Omega)$, that the 
matrix-valued function 
$\boldsymbol{K} \in \bigl[L^\infty(\Omega)\bigr]^{d \times d}$ is symmetric with 
smallest eigenvalue uniformly bounded from below, and that 
$\operatorname{div} \boldsymbol{b}\in L^\infty(\Omega)$ with 
\begin{equation*}
- \operatorname{div}\bigl(\boldsymbol{b}\bigr) + 2 \, c \geq 0 
\quad\text{a.e. in } \Omega.
\end{equation*}
Under these assumptions, the bilinear form \(b(\cdot, \cdot)\) 
is coercive and bounded on \(\XX\) with respect to the energy norm 
\(\enorm{v} \coloneqq a(v, v)^{1/2}\),
i.e., there exist $0 < C_{\textup{ell}} \leq  C_{\textup{bnd}}$ such that
\begin{equation}\label{eq:coercivity_continuity}
	\abs{b(v, w)} \le C_{\textup{bnd}} \, \enorm{v} \, \enorm{w}
	\quad \text{and} \quad
	b(v, v) \ge C_{\textup{ell}} \, \enorm{v}^2
	\quad \text{for all } v, w \in \XX.
\end{equation}
The Lax--Milgram lemma yields existence and uniqueness of the 
solution \(u^\star\) to~\eqref{eq:weak_formulation}.
For the well-posedness of the standard residual-based error estimator (see~\eqref{eq:definition_eta} below), we will require additional regularity $\vec{f}|_T \in [H^1(T)]^d$ and $\vec{K}|_T \in [W^{1,\infty}(T)]^{d \times d}$ for all $T \in \TT_0$ with $\TT_0$ being the initial triangulation of $\Omega$ introduced in the following section.

\subsection{Finite element method}\label{section:adaptive_FEM}
Let \(\TT_0\) be a conforming initial triangulation of \(\Omega\) consisting of 
compact simplices. The adaptive mesh refinement employs newest vertex 
bisection (NVB),
where we refer to \cite{stevenson2008} for NVB with admissible initial 
triangulation and dimension \(d \ge 2\), to
\cite{Karkulik2013a} for NVB with general \(\mathcal{T}_0\) in the case \(d=2\), 
to the recent work \cite{dgs2023} for NVB with general $\mathcal{T}_0$ in any dimension
\(d \ge 2\), and to \cite{AFFKP13} for NVB in the case \(d=1\). 
For any triangulation \(\TT_H\) and any subset \(\mathcal{M}_H \subseteq \TT_H\), 
we define \(\TT_h \coloneqq \mathtt{refine}( \TT_H, \mathcal{M}_H )\) as the 
coarsest refinement of \(\TT_H\) by newest vertex bisection such that at least 
all elements in \(\mathcal{M}_H\) are refined, i.e., $\MM_H \subseteq \TT_H \backslash \TT_h$. 
The notation \(\TT_h \in \T(\TT_H)\) abbreviates that \(\TT_h\) can be obtained from 
\(\TT_H\) by a finite number of newest vertex bisection steps. To each 
triangulation \(\TT_H\) we associate the finite element space 
\begin{equation}\label{eq:fem_spaces}
	\XX_H \coloneqq 
\mathbb{S}^p_0(\TT_H) \coloneqq \mathbb{S}^p(\TT_H) \cap \XX
\end{equation}
consisting of 
globally continuous $\TT_H$-piecewiese polynomials of degree at most 
\(p \in \N\). Note that 
\(\TT_h \in \T(\TT_H)\) implies nestedness of the associated finite element spaces \(\XX_H \subseteq \XX_h\). 

The finite element discretization 
for \eqref{eq:weak_formulation} reads as follows: Find \(u_H^\star \in \XX_H\) 
such that
\begin{equation}\label{eq:discrete_formulation}
	b(u_H^\star, v_H) = F(v_H)
	\quad \text{for all } v_H \in \XX_H.
\end{equation}
The Lax--Milgram lemma ensures the existence and uniqueness of the solution 
\(u_H^\star\) to \eqref{eq:discrete_formulation} and the solution is the 
quasi-best approximation in \(\XX_H\) to the solution \(u^\star\) of 
\eqref{eq:weak_formulation} in the sense that there exists a constant 
\(1 \le C_{\textup{Céa}} \le C_{\textup{bnd}} / C_{\textup{ell}}\) such that
\begin{equation}\label{eq:cea}
	\enorm{u^\star - u_H^\star} \le C_{\textup{Céa}}
	\min_{v_H \in \XX_H} \enorm{u^\star - v_H},
\end{equation}
where we note that $\Ccea \rightarrow 1$ as $\enorm{u^{\star}-u_H^\star} \rightarrow 0$; see~\cite{bhp2017}.

\subsection{A-posteriori error estimation}\label{section:a_posteriori}
Based on the additional regularity of the data, we consider the residual error estimator~\(\eta_H(\cdot)\)
defined, for \(T \in \mathcal{T}_H\) and \(v_H \in \mathcal{X}_H\), by
\begin{subequations}\label{eq:definition_eta}
	\begin{equation}\label{eq:definition_eta:a}
		\begin{aligned}
			\eta_H(T, v_H)^2
			& \coloneqq
			|T|^{2/d} \, \norm{\operatorname{div}(\boldsymbol{K} 
			\nabla v_H - \boldsymbol{f}) - \boldsymbol{b} 
			\cdot \nabla v_H - c \, v_H + f
			}_{L^2(T)}^2
			\\& \qquad
			+
			|T|^{1/d} \, \norm{\lbrack\!\lbrack (\boldsymbol{K}
			\nabla v_H - \boldsymbol{f}) \cdot n \rbrack\!\rbrack
			}_{L^2(\partial T \cap \Omega)}^2,
		\end{aligned}
	\end{equation}
	where \(\lbrack\!\lbrack \cdot \rbrack\!\rbrack\) denotes the 
	jump over the $(d-1)$-dimensional facets of $T$. 
	Clearly, the well-posedness of~\eqref{eq:definition_eta:a} requires the additional regularity of the data $\vec{K}$ and $\vec{f}$ stated above.
	For any subset $\mathcal{U}_H \subseteq \mathcal{T}_H$ and all 
	$v_H \in \mathcal{X}_H$, we define the global error estimator by
	\begin{equation}
		\eta_H(v_H)
		\coloneqq
		\eta_H(\mathcal{T}_H, v_H)
		\quad \text{with} \quad
		\eta_H(\mathcal{U}_H, v_H)
		\coloneqq \Big( \sum_{T \in \mathcal{U}_H} \eta_H(T, v_H)^2
		\Big)^{1/2}.
	\end{equation}
\end{subequations}
It is well-known that the error estimator satisfies the
axioms of adaptivity; see~\cite{cfpp2014}.
\begin{proposition}[axioms of adaptivity]\label{prop:axioms}
	There exist constants $C_{\textup{stab}}, C_{\textup{rel}}, 
	C_{\textup{drel}}, C_{\textup{mon}} > 0$ and $0 <
		q_{\textup{red}} < 1$ such that the following properties are 
		satisfied for any triangulation $\mathcal{T}_H \in \T(\TT_0)$, any refinement $\mathcal{T}_h \in
		\mathbb{T}(\mathcal{T}_H)$, and arbitrary 
		\(v_H \in \mathcal{X}_H$, $v_h \in \mathcal{X}_h\):
	\begin{enumerate}[leftmargin=4em]
		\renewcommand{\theenumi}{A\arabic{enumi}}
		\bf
		\item[(A1)]\refstepcounter{enumi}\label{axiom:stability} 
		\textit{stability:}
		\quad \rm
			$|\eta_h(\mathcal{T}_h \cap \mathcal{T}_H, v_h) 
			- \eta_H(\mathcal{T}_h \cap \mathcal{T}_H, v_H)| 
			\le C_{\textup{stab}} \, \enorm{v_h - v_H}$;
			\bf
		\item[(A2)]\refstepcounter{enumi}\label{axiom:reduction} 
		\textit{reduction:}
			\quad \rm
		$\eta_h(\mathcal{T}_h \backslash \mathcal{T}_H, v_H) 
		\le q_{\textup{red}} \, \eta_H(\mathcal{T}_H \backslash 
		\mathcal{T}_h, v_H)$;
			\bf
		\item[(A3)]\refstepcounter{enumi}\label{axiom:reliability}
			\textit{reliability:} \quad \rm
			$\enorm{u^\star - u_H^\star} \le C_{\textup{rel}} \,
			\eta_H(u_H^\star)$;

			\renewcommand{\theenumi}{A3$^{+}$}
			\bf
		\item[(A3$^+$)]\refstepcounter{enumi}
			\label{axiom:discrete_reliability}
			\it
			\textbf{discrete reliability:}
				\quad
				\rm
		$
			\enorm{u_h^\star - u_H^\star}
			\le
			C_{\textup{drel}} \, 
			\eta_H(\mathcal{T}_H \backslash \mathcal{T}_h, u_H^\star);
		$

		\renewcommand{\theenumi}{QM}
		\bf
		\item[(QM)]\refstepcounter{enumi}\label{eq:quasi-monotonicity}
			\textit{quasi-monotonicity:} \quad \rm
			$\eta_h(u_h^\star) \le C_{\textup{mon}} \,
			\eta_H(u_H^\star)$.

	\end{enumerate}
	The constant $C_{\textup{rel}}$ depends only on uniform shape 
	regularity, the dimension $d$, the ellipticity constant $\Cell$, and $\norm{\vec{K}}_{\infty}+\norm{\vec{b}}_{\infty}+\norm{c}_{\infty}$, whereas $C_{\textup{stab}}$ and
	\(C_{\textup{drel}}\) additionally depend on the polynomial degree $p$. 
	Moreover, there hold the bounds $q_{\textup{red}} \leq 2^{-1/(2d)}$ and $C_{\textup{mon}}
		\le \min\{ 1 +
			C_{\textup{stab}}(1+C_{\textup{Céa}}) \, 
			C_{\textup{rel}} \, , \, 1 
			+ C_{\textup{stab}} \, C_{\textup{drel}}\}$.
		\qed
\end{proposition}

In addition to the estimator properties in Proposition~\ref{prop:axioms}, 
we recall the following quasi-orthogonality
result from~\cite{feischl2022} as one cornerstone of the analysis 
in this paper.

\begin{proposition}[validity of quasi-orthogonality]\label{prop:orthogonality}
	There exist
	\(C_{\textup{orth}} > 0\) and \(0 < \delta \le 1\) such
	that
	the following holds: For any sequence 
	$\mathcal{X}_\ell \subseteq \mathcal{X}_{\ell+1} \subset
		H^1_0(\Omega)$ of nested finite-dimensional subspaces, 
		the corresponding Galerkin
	solutions $u_\ell^\star \in \mathcal{X}_\ell$ to 
	\eqref{eq:discrete_formulation} satisfy
	\begin{enumerate}[leftmargin=2.7em]
		\renewcommand{\theenumi}{A\arabic{enumi}}
		\setcounter{enumi}{3}
		\bf
		\item[(A4)]\refstepcounter{enumi}\label{axiom:orthogonality} 
		\textit{quasi-orthogonality}: 
		\[
			\sum_{\ell' = \ell}^{\ell + N} 
			\enorm{u^\star_{\ell'\!+\!1} \!-
			u^\star_{\ell'}}^2
			\!\le\!
			\Corth (N+1)^{1-\delta} 
			\enorm{u^\star \!- u^\star_\ell}^2
			\quad \textnormal{for all }  \ell, N \in \mathbb{N}_0.
			\]
	\end{enumerate}
	Here, $C_{\textup{orth}}$ and \(\delta\) depend only on the
	dimension $d$, the elliptic bilinear form $b(\cdot,\cdot)$, and the 
	chosen norm
	$\enorm{\cdot}$, but are 
	independent of the spaces $\mathcal{X}_\ell$. \qed
\end{proposition}

It remains to treat the discrete 
problem \eqref{eq:discrete_formulation}. 
Note that, in the case of a symmetric setting \(b(\cdot, \cdot) = 
a(\cdot, \cdot)\),
optimal solvers are well-established; see, e.g.,
\cite{cnx2012} for an optimally preconditioned conjugate gradient method (CG),
\cite{wz2017, imps2022} for optimal local
geometric multigrid methods, and the recent work~\cite{hmp26} for multigrid preconditioners for (generalized) CG. 
In the next section, we revisit the literature on 
preconditioned GMRES methods and propose
an optimally preconditioned GMRES method for the numerical solution of the 
non-symmetric discrete problem~\eqref{eq:discrete_formulation}.


\section{Preconditioned GMRES method (PGMRES)} \label{section:preconditioned_gmres}
\subsection{Abstract setting}
Fix \(\ell \in \N\) and let us consider the discrete 
problem~\eqref{eq:discrete_formulation}: let \(\varphi_{\ell,1}, \ldots, \varphi_{\ell,N_{\ell}}\) 
be the basis functions of \(\XX_\ell\) with 
\(N_{\ell}\coloneqq \dim \XX_\ell\). We define the Galerkin matrix 
\(\vec{B}_\ell \in \R^{N_{\ell}\times N_{\ell}}\) associated with the 
bilinear form \(b(\cdot, \cdot)\) by
$(\vec{B}_\ell)_{jk} \coloneqq b(\varphi_{\ell,k} , \varphi_{\ell,j})$ for all 
$j, k = 1, \ldots, N_{\ell}$. Similarly, discretizing the right-hand side leads 
to $\vec{d}_\ell \in\R^{N_{\ell}}$. The arising linear 
problem from the finite element discretization aims to find 
$\vec{x}^\star_\ell \in\R^{N_{\ell}}$ solving
\begin{equation} \label{eq:lin_system}
	\vec{B}_\ell \, \vec{x}^\star_\ell = \vec{d}_\ell.
\end{equation}
Let $\vec{P}_\ell \in \R^{N_{\ell}\times N_{\ell}}$ be an SPD
preconditioner of $\vec{B}_{\ell}$, i.e., $\vec{P}_{\ell} \approx \vec{B}_{\ell}^{-1}$ is symmetric 
and positive definite.
Note that
\((\cdot, \cdot )_{\vec{P}^{-1}_\ell} \coloneqq 
(\vec{P}^{-1}_\ell \cdot, \cdot)_2\) then defines a scalar product on 
\(\R^{N_{\ell}}\). 
We suppose that $\vec{P}_{\ell}$ is optimal in the following sense: first, the matrix-vector multiplication with $\vec{P}_{\ell}$ can be realised in linear complexity; second, the  discrete norm induced by the inverse is equivalent to the PDE-related energy norm $\enorm{\cdot}$ defined in Section~\ref{section:model_problem}, i.e.,  there exists a constant \(C_{\textup{pre}}\ge 1\) independent of the discretization parameters, i.e., the local mesh size $h$ and the polynomial degree $p$, such 
that 
\begin{equation} \label{eq:equivalence_preconditioner_energy_norm}
	C_{\textup{pre}}^{-1} \, \enorm{v_\ell} 
	\le 
	\norm{\vec{x}_\ell}_{\vec{P}^{-1}_\ell} 
	\coloneq (\vec{x}_\ell, \vec{x}_\ell )^{1/2}_{\vec{P}^{-1}_\ell}
	\le  C_{\textup{pre}} \, \enorm{v_\ell} \quad \text{for all }
v_\ell = \sum_{j=1}^{N_{\ell}} (\vec{x}_\ell)_{j} \, 
\varphi_{\ell,j} \in \XX_\ell.
\end{equation}
We now formulate the optimally preconditioned GMRES algorithm for~\eqref{eq:lin_system}. 
It is derived from~\cite[Algorithm 3.1]{SarkisSzyld07}, where all the scalar products are replaced by the preconditioner inner product $(\cdot, \cdot )_{\vec{P}^{-1}_\ell}$.
Moreover, a restart strategy is included to ensure that the computational cost of the algorithm is linear in the number of degrees of freedom.
To make the computations explicit at step $k$ of the algorithm, we compute and store the Arnoldi basis vectors
denoted by $\vec{v}^j$ as well as the
unpreconditioned Arnoldi basis vectors 
denoted by $\widetilde{\vec{v}}^j$ for the current run, i.e., these vectors are deleted after every restart.

\begin{algorithm}[optimally preconditioned GMRES method with restart]\label{alg:optimally_preconditioned_gmres} 

\textbf{Input:} Matrix $\vec{B}_{\ell}$, preconditioner matrix 
	$\vec{P}_{\ell}$, 
	right-hand side vector $\vec{d}_{\ell}$, initial guess $\vec{x}_{\ell}^0$, absolute tolerance function \(\Lambda: \R^{N_{\ell}}\to \R_{\ge 0}\), PGMRES restart length $k_{\max} \in \N$.
\begin{enumerate}[label = \textnormal{(\alph*)}]
	\item Compute $\vec{r}_{\ell}^0 \coloneqq \vec{d}_{\ell}-\vec{B}_{\ell}\vec{x}_{\ell}^0, \vec{s}_{\ell}^0\coloneqq
	\vec{P}_{\ell}\vec{r}_{\ell}^0$, 
	$\widetilde{\vec{v}}^1\coloneqq \vec{r}_{\ell}^0/\norm{\vec{s}_{\ell}^0
	}_{\vec{P}^{-1}_{\ell}} = \vec{r}_{\ell}^0/(\vec{s}_{\ell}^0, \vec{r}_{\ell}^0)_{2}^{1/2}$,
	$\vec{v}^1\coloneqq \vec{P}_{\ell}
	\widetilde{\vec{v}}^1$.
	\item For all \(k = 1, 2,3,\ldots\) repeat the following steps
	\textnormal{(i)--(vi)} \textbf{until}
	\begin{equation}\label{eq:single:termination}
			\norm{\vec{s}_{\ell}^k}_{\vec{P}^{-1}_{\ell}}\leq 
		\Lambda(\vec{x}_{\ell}^k) \quad \text{and then set} \quad \underline{k}[\ell] \coloneqq k.
		\end{equation}
	\begin{enumerate}[label=\textnormal{(\roman*)}]
		\item Compute $\widetilde{\vec{w}} \coloneqq
		\vec{B}_{\ell}\vec{v}^k$, $\vec{w} \coloneqq 
		\vec{P}_{\ell} \widetilde{\vec{w}}$ and set $K \coloneqq \mod{k-1}{k_{\max}}+1$ as well as the restart index $R \coloneqq \lfloor k / k_{\max}\rfloor$.
		\item \label{alg:gmres:ii} For all $j=1, \dots, K$ compute $\vec{H}_{jK}
			\coloneqq 
			\dual{ \vec{w}}{\vec{v}^{k+j-K}}_{\vec{P}^{-1}_{\ell}}
			=
			\dual{ \widetilde{\vec{w}} }{
				\vec{v}^{k+j-K} }_2$
		and
		\begin{equation*}
			\widetilde{\vec{w}} \coloneqq \widetilde{\vec{w}} 
			- \vec{H}_{jK} \widetilde{\vec{v}}^{k+j-K} \quad \text{and} \quad
			\vec{w} \coloneqq \vec{w} 
			- \vec{H}_{jK} \vec{v}^{k+j-K}.
		\end{equation*}
		\item Compute $\vec{H}_{K+1,K} \coloneqq
		\norm{\vec{w}}_{\vec{P}^{-1}_{\ell}} =
		\dual{\widetilde{\vec{w}} }{ \vec{w}}^{1/2}_2$.
		\item\label{alg:PGMRES:min} Compute
		$	\vec{y} \coloneqq \argmin_{\vec{z}\in\R^{K}}  \big\lVert
			\,\norm{\vec{s}_{\ell}^{R k_{\max}}}_{\vec{P}^{-1}_{\ell}} \vec{e}^1 
			- \vec{H}_{[1:K+1,1:K]}\vec{z} \big\rVert_2.$
		\item Compute $\vec{x}_{\ell}^k \coloneqq 
		\vec{x}_{\ell}^{R k_{\max}}+\sum_{j=1}^{K} 
		\vec{y}_j\vec{v}^{k+j-K}$, $\vec{r}_{\ell}^k 
		 \coloneqq 
		\vec{d}_{\ell} - \vec{B}_{\ell} \vec{x}_{\ell}^k$, 
		$\vec{s}_{\ell}^k \coloneqq \vec{P}_{\ell} 
		\vec{r}_{\ell}^k$,
		and $\norm{\vec{s}_{\ell}^k}_{\vec{P}^{-1}_{\ell}} = (\vec{s}_{\ell}^k, \vec{r}_{\ell}^k)_2^{1/2}$.
		\item If $\mod{k}{k_{\max}}=0$, then define $\widetilde{\vec{v}}^{k+1} \coloneqq \vec{r}_{\ell}^k / \norm{\vec{s}_{\ell}^k}_{\vec{P}^{-1}_\ell}$, $\vec{v}^{k+1} \coloneqq \vec{s}_{\ell}^k/ \norm{\vec{s}_{\ell}^k}_{\vec{P}^{-1}_\ell}$, and delete $\vec{H}, \vec{v}^j$, and $\widetilde{\vec{v}}^j$ for $j \leq k$; Else, compute $\widetilde{\vec{v}}^{k+1} 
		 \coloneqq \widetilde{\vec{w}}/\vec{H}_{k+1,k}$,
		$\vec{v}^{k+1} 
		\coloneqq \vec{w}/\vec{H}_{k+1,k}$.
\end{enumerate}
\end{enumerate}
\textbf{Output:} Approximations $\vec{x}_{\ell}^{k}$, norms of preconditioned residuals 
$\norm{\vec{s}_{\ell}^{k}}_{\vec{P}^{-1}_{\ell}}$, and index $\underline{k}[\ell]$.
\end{algorithm}

\begin{remark}[notation of Algorithm~\ref{alg:optimally_preconditioned_gmres}]
Define $K \coloneqq \mod{k-1}{k_{\max}}+1$ as 
\begin{equation*}
	k \mapsto \begin{cases}
			k_{\max} & \text{for } \mod{k}{k_{\max}} = 0, \\
		\mod{k}{k_{\max}} & \text{else},
	\end{cases}
\end{equation*}
and thus yields the number of GMRES steps since the last restart.
Moreover, the quantity $R \coloneq \lfloor k/k_{\max}\rfloor$ gives the number of restarts performed so far.
In step~\ref{alg:gmres:ii}, the orthogonalization loop is carried out over the last $K$ Arnoldi basis vectors $\vec{v}^k$. 
This is realized through the index $k+j-K$, where $j=1,\dots,K$.
Finally, $\vec{x}_{\ell}^{Rk_{\max}}$ denotes the initial guess of the current GMRES run, and $\vec{s}_{\ell}^{Rk_{\max}}$ the corresponding initial preconditioned residual.
\end{remark}

\begin{remark}[computational complexity of PGMRES with restart]\label{label:remark_cost_PGMRES}
For every $k$, Algorithm~\ref{alg:optimally_preconditioned_gmres} performs two applications of $\vec{B}_{\ell}$ as well as two applications of $\vec{P}_{\ell}$, which due to sparsity of $\vec{B}_{\ell}$ and optimality assumptions on the preconditioner $\vec{P}_{\ell}$ can be performed with linear cost $\OO(N_{\ell})$.
Since $K \leq k_{\max}$ the $j$-loop is bounded and the calculation of $\vec{y}$ in step~\ref{alg:PGMRES:min} is of constant cost with respect to $k$ and the degrees of freedom.
Therefore, one step of optimal preconditioned GMRES with restart is of cost $\OO(N_{\ell})$.
\end{remark}

The analysis of Algorithm~\ref{alg:optimally_preconditioned_gmres} relies on 
an equivalence of the discrete and level-induced
norm \(\norm{\vec{s}^k_\ell}_{\vec{P}^{-1}_\ell}\) of the $k$-th residual vector 
and the functional energy norm 
\(\enorm{u_{\ell}^\star - u_{\ell}^k}\) of the algebraic error, where $u_{\ell}^k \coloneqq \sum_{j=1}^{N_{\ell}} (\vec{x}^k_\ell)_j \varphi_{\ell,j}$ and $\vec{x}_{\ell}^k$ is the $k$-th iterate of Algorithm~\ref{alg:optimally_preconditioned_gmres}.
First, let us show a preliminary equivalence result on the preconditioned norm 
\(\norm{\cdot}_{\vec{P}^{-1}_\ell}\).

\begin{lemma}[optimality of preconditioner in discrete norm]
	\label{lem:equivalence_preconditioner_norm}
  For any \(\vec{x}_\ell \in \R^{N_{\ell}}\) and \(v_\ell \coloneqq \sum_{j=1}^{N_{\ell}} (\vec{x}_\ell)_j 
  \varphi_{\ell,j} \in \XX_\ell\), there hold the estimates
 
\begin{subequations}\label{eq:equivalence_preconditioner_norm}
	\begin{equation}\label{eq:lower}
		C_{\textup{pre}}^{-1} \, \norm{ \vec{x}_\ell }_{\vec{P}^{-1}_\ell}  \eqreff{eq:equivalence_preconditioner_energy_norm}\leq \enorm{v_\ell}\le C_{\textup{ell}}^{-1} \, C_{\textup{pre}} \norm{ \vec{P}_\ell \vec{B}_\ell \, 
		\vec{x}_\ell }_{\vec{P}^{-1}_\ell} 
	\end{equation}
	and
	\begin{equation}\label{eq:upper}
		C_{\textup{bnd}}^{-1} \, C_{\textup{pre}}^{-1}  \norm{ \vec{P}_\ell \vec{B}_\ell \, 
		\vec{x}_\ell }_{\vec{P}^{-1}_\ell}  \leq  \enorm{v_\ell} \eqreff{eq:equivalence_preconditioner_energy_norm}\leq C_{\textup{pre}}  \, \norm{ \vec{x}_\ell }_{\vec{P}^{-1}_\ell}.  
	\end{equation}
\end{subequations}
	In particular, all estimates are $h$- and $p$-robust.
\end{lemma}

\begin{proof}
	The estimates are trivial for \(\vec{x}_\ell = 0\). For the case \(\vec{x}_\ell \neq 0\), note that 
	\begin{equation*}
			\norm{ \vec{P}_\ell \vec{B}_\ell 
			\vec{x}_\ell }_{\vec{P}^{-1}_\ell}
			= 
			\sup_{\vec{y}_\ell \in \R^{N_{\ell}} \setminus \{0\}} 
			\frac{( \vec{P}_\ell \vec{B}_\ell 
			\vec{x}_\ell, \vec{y}_\ell)_{\vec{P}^{-1}_\ell}
			}{\norm{ \vec{y}_\ell }_{\vec{P}^{-1}_\ell}}
			\ge
			\frac{( \vec{P}_\ell \vec{B}_\ell 
			\vec{x}_\ell, \vec{x}_\ell)_{\vec{P}^{-1}_\ell}
			}{\norm{ \vec{x}_\ell }_{\vec{P}^{-1}_\ell}}
			=
			\frac{( \vec{B}_\ell  \vec{x}_\ell, 
			\vec{x}_\ell)_2}{\norm{ \vec{x}_\ell }_{\vec{P}^{-1}_\ell}}.
	\end{equation*}
	Moreover,~\eqref{eq:coercivity_continuity} and 
	\eqref{eq:equivalence_preconditioner_energy_norm} result in
	\begin{equation*}
		\frac{( \vec{B}_\ell  \vec{x}_\ell, 
			\vec{x}_\ell)_2}{\norm{ \vec{x}_\ell }_{\vec{P}^{-1}_\ell}} =
			\frac{b(v_\ell, v_\ell)}{\norm{ \vec{x}_\ell }_{\vec{P}^{-1}_\ell}}
			\eqreff{eq:coercivity_continuity}
			\ge
			\frac{C_{\textup{ell}} \, \enorm{v_\ell}^2}{\norm{ \vec{x}_\ell }_{\vec{P}^{-1}_\ell}}
			\eqreff{eq:equivalence_preconditioner_energy_norm}
			\ge
      \frac{C_{\textup{ell}}}{C_{\textup{pre}}} 
      \enorm{v_\ell} 
	  \eqreff{eq:equivalence_preconditioner_energy_norm}{\ge}
			  \frac{C_{\textup{ell}}}{C_{\textup{pre}}^2}  \, \norm{ 
      \vec{x}_\ell }_{\vec{P}^{-1}_\ell}.
	\end{equation*}
	The combination of the previous two formulas establishes~\eqref{eq:lower}.
	For~\eqref{eq:upper}, set \(\vec{y}_\ell 
	\coloneqq \vec{P}_\ell \vec{B}_\ell 
	\vec{x}_\ell \in \R^{N_{\ell}}\) and 
	consider \(w_\ell \coloneqq \sum_{j=1}^{N_{\ell}} (\vec{y}_\ell)_j 
	\varphi_{\ell,j} \in \XX_\ell\). Then, \eqref{eq:coercivity_continuity} and
	\eqref{eq:equivalence_preconditioner_energy_norm} result in
	\begin{equation*}
		\begin{aligned}
			\norm{ \vec{y}_\ell }_{\vec{P}^{-1}_\ell}^2
			&=
			( \vec{P}_\ell \vec{B}_\ell \vec{x}_\ell, 
			\vec{y}_\ell)_{\vec{P}^{-1}_\ell}
			=
			( \vec{B}_\ell \vec{x}_\ell, 
			\vec{y}_\ell)_{2}
			\\
			&=
			b(v_\ell, w_\ell)
			\eqreff{eq:coercivity_continuity}
			\le
			C_{\textup{bnd}} \, \enorm{v_\ell} \, \enorm{w_\ell}
			\eqreff{eq:equivalence_preconditioner_energy_norm}
			\le
      C_{\textup{bnd}} \, C_{\textup{pre}} \, \enorm{v_\ell} \, \norm{ \vec{y}_\ell }_{\vec{P}^{-1}_\ell}.
		\end{aligned}
	\end{equation*}
  Rearranging the estimate and emplyoing~\eqref{eq:equivalence_preconditioner_energy_norm}, we prove~\eqref{eq:upper} and conclude the proof.
\end{proof}

Lemma~\ref{lem:equivalence_preconditioner_norm} can be used to establish 
the desired equivalence of the energy norm of the algebraic error and the 
algebraic residual norm.

\begin{lemma}[equivalence of discrete norm and functional algebraic error]
	\label{lem:equivalence_residual_norm}
	Let $\vec{x}_{\ell}^k$ denote the $k$-th iterate computed in Algorithm~\ref{alg:optimally_preconditioned_gmres}, $\vec{s}_{\ell}^k$ the preconditioned residual and $u_{\ell}^k \coloneqq \sum_{j=1}^{N_{\ell}} (\vec{x}^k_\ell)_j \varphi_{\ell,j} \in \XX_{\ell}$ the corresponding function.
	Then, there holds the equivalence
	\begin{equation}\label{eq:equivalence_residual_norm}
    C_{\textup{ell}} \, C_{\textup{pre}}^{-1} \,
    \enorm{ u_\ell^\star - u_\ell^k}
    \le
		\norm{ \vec{s}^k_\ell }_{\vec{P}^{-1}_\ell} 
        \le  C_{\textup{bnd}} \, C_{\textup{pre}}\, 
		\enorm{ u_\ell^\star - u_\ell^k }.
	\end{equation}
	In particular, the equivalence is $h$- and $p$-robust.
\end{lemma}

\begin{proof}
	The definition of $\vec{s}_{\ell}^k$ in Algorithm~\ref{alg:optimally_preconditioned_gmres} and~\eqref{eq:lin_system} yield
	\begin{equation*}
			\norm{ \vec{s}^k_\ell }_{\vec{P}^{-1}_\ell}
		=
		\norm{ \vec{P}_\ell (\vec{d}_\ell 
		- \vec{B}_\ell \vec{x}^k_\ell) }_{\vec{P}^{-1}_\ell}
        \eqreff{eq:lin_system}
		 =
		\norm{ \vec{P}_\ell 
		(\vec{B}_\ell \vec{x}^\star_\ell 
		- \vec{B}_\ell \vec{x}_\ell^k) }_{\vec{P}^{-1}_\ell}.
	\end{equation*}
    Combining this with 
    \eqref{eq:equivalence_preconditioner_norm} 
    and $u_{\ell}^{\star}-u_{\ell}^k = \sum_{j=1}^{N_{\ell}} (\vec{x}_{\ell}^{\star}-\vec{x}_{\ell})_j \varphi_{\ell,j}$ 
    finishes the proof.
\end{proof}

\begin{proposition}[discrete contraction of optimally preconditioned GMRES]
	\label{lem:contraction_gmres}
	Let $k_{\max} \in \N$ be abitrary. Suppose formally that $\Lambda < 0$ so that Algorithm~\ref{alg:optimally_preconditioned_gmres} does not terminate.
	Then, the optimally preconditioned GMRES method from 
	Algorithm~\ref{alg:optimally_preconditioned_gmres} 
	guarantees
	\begin{equation}\label{eq:contr_gmres}
		\norm{ \vec{s}^{k}_\ell }_{\vec{P}^{-1}_\ell} 
		\le q_{\textup{alg}} \, 
		\norm{ \vec{s}^{k-1}_\ell }_{\vec{P}^{-1}_\ell} \quad \text{for all } k \in \N,
	\end{equation}
	with \(0 < q_{\textup{alg}}  < 1\) depending only on \(C_{\textup{bnd}}\), 
	\(C_{\textup{ell}}\), and \(C_{\textup{pre}}\). 
\end{proposition}

\begin{proof}
By construction, the preconditioned GMRES method is a Krylov subspace method which 
minimizes residuals with respect to the preconditioned inner product over the 
affine Krylov space generated by the preconditioned matrix 
$\vec{P}_\ell \vec{B}_\ell$.
Note that the Krylov spaces are reset at every restart of Algorithm~\ref{alg:optimally_preconditioned_gmres}. 
Recalling $K = \mod{k-1}{k_{\max}}+1 \geq 1$ and the parameter $R$ that counts the number of restarts and noting that $\vec{s}_{\ell}^{Rk_{\max}}$ is the initial residual of the current GMRES run, it hence follows
\begin{align}\label{eq:GMRES_residualpresentation}
\norm{ \vec{s}^k_\ell }_{\vec{P}^{-1}_\ell} 
= \min_{p\in \mathbb{P}_{K}, \, 
p(0)=1} \norm{ p(\vec{P}_\ell 
\vec{B}_\ell) \vec{s}_\ell^{Rk_{\max}} }_{\vec{P}^{-1}_\ell},
\end{align}
where $\vec{s}_\ell^k 
= \vec{P}_\ell(\vec{d}_\ell 
-\vec{B}_\ell\vec{x}_\ell^k) 
= \vec{P}_\ell\vec{B}_\ell(\vec{x}_\ell
-\vec{x}_\ell^k)$ and $
\mathbb{P}_{K}$ denotes the space of polynomials (in one variable) of 
degree $\leq K$; see~\cite{EES83,SarkisSzyld07} and references therein. 
This yields
\begin{align}\label{eq:residual_1}
    \norm{ \vec{s}_\ell^{k-1} }_{\vec{P}^{-1}_\ell} 
	= \min_{p\in\mathbb{P}_{K-1},p(0)=1} \norm{ p(\vec{P}_\ell
	\vec{B}_\ell) \vec{s}_\ell^{Rk_{\max}} }_{\vec{P}^{-1}_\ell}
    = \norm{ p_{K-1}(\vec{P}_\ell
	\vec{B}_\ell) \vec{s}_\ell^{Rk_{\max}} }_{\vec{P}^{-1}_\ell},
\end{align}
Let $\alpha \in\R$ (to be fixed below) and define the polynomial 
$\widetilde p$ by $\widetilde p(t) = (1-\alpha t)p_{K-1}(t)$, $t\in\R$.
Then, $\widetilde p\in \mathbb{P}_{K}$. 
Furthermore, $\widetilde p(0) = (1-\alpha 0)p_{K-1}(0) = 1$. Therefore, 
\begin{align*}
   \norm{ &\vec{s}_\ell^k }_{\vec{P}^{-1}_\ell} 
	= \,\min_{p\in\mathbb{P}_{K},p(0)=1} 
	\norm{ p(\vec{P}_\ell\vec{B}_\ell)
	\vec{s}_\ell^{Rk_{\max}} }_{\vec{P}^{-1}_\ell}
    \leq \norm{ \widetilde p(\vec{P}_\ell \vec{B}_\ell)
	\vec{s}_\ell^{Rk_{\max}} }_{\vec{P}^{-1}_\ell} \\
    &
    = \,\norm{ (1-\alpha\vec{P}_\ell\vec{B}_\ell)
	p_{K-1}(\vec{P}_\ell\vec{B}_\ell)
	\vec{s}_\ell^{Rk_{\max}} }_{\vec{P}^{-1}_\ell} \\
   & \leq \,\norm{ (1-\alpha\vec{P}_\ell
	\vec{B}_\ell) }_{\vec{P}^{-1}_\ell} 
	\norm{ p_{K-1}(\vec{P}_\ell\vec{B}_\ell)
	\vec{s}_\ell^{Rk_{\max}} }_{\vec{P}^{-1}_\ell} 
	\eqreff*{eq:residual_1}= \, \norm{ (1-\alpha\vec{P}_\ell
	\vec{B}_\ell) }_{\vec{P}^{-1}_\ell} 
	\norm{ \vec{s}_\ell^{k-1} }_{\vec{P}^{-1}_\ell}.
\end{align*}
To prove
$\norm{(1-\alpha\vec{P}_\ell
\vec{B}_\ell)}_{\vec{P}^{-1}_\ell}\leq q_\mathrm{alg}$, let
$\vec{x}_\ell\in\R^{N_{\ell}}$ with 
$\norm{ \vec{x}_\ell }_{\vec{P}^{-1}_\ell}=1$. 
Then, there holds
\begin{align*}
    \norm{ (1-\alpha\vec{P}_\ell\vec{B}_\ell)
	\vec{x}_\ell }_{\vec{P}^{-1}_\ell}^2
	&= \norm{ \vec{x}_\ell }_{\vec{P}^{-1}_\ell}^2 
	- 2\alpha (\vec{P}_\ell
	\vec{B}_\ell\vec{x}_\ell, 
	\vec{x}_\ell )_{\vec{P}^{-1}_\ell} + \alpha^2 
	\norm{ \vec{P}_\ell \vec{B}_\ell
	\vec{x}_\ell }_{\vec{P}^{-1}_\ell}^2.
\end{align*}
Viewing the right-hand side as a quadratic polynomial in $\alpha$, 
it attains its minimum for 
$\alpha = (\vec{P}_\ell\vec{B}_\ell\vec{x}_\ell,\vec{x}_\ell )_{\vec{P}^{-1}_\ell}/\norm{ \vec{P}_\ell \vec{B}_\ell
  \vec{x}_\ell }_{\vec{P}^{-1}_\ell}^2$.
  This choice of $\alpha$ leads to
  \begin{align*}
    \norm{ \vec{x}_\ell }_{\vec{P}^{-1}_\ell}^2 - 2\alpha (\vec{P}_\ell\vec{B}_\ell\vec{x}_\ell, \vec{x}_\ell )_{\vec{P}^{-1}_\ell} + \alpha^2 \norm{ \vec{P}_\ell \vec{B}_\ell\vec{x}_\ell }_{\vec{P}^{-1}_\ell}^2
    = 1-\frac{(\vec{P}_\ell\vec{B}_\ell\vec{x}_\ell, \vec{x}_\ell )_{\vec{P}^{-1}_\ell}^2}{\norm{ \vec{P}_\ell \vec{B}_\ell\vec{x}_\ell }_{\vec{P}^{-1}_\ell}^2}.
  \end{align*}
  With \(v_\ell \coloneqq \sum_{j=1}^{N_{\ell}} (\vec{x}_\ell)_j 
  \varphi_{\ell,j} \in \XX_\ell\),~\eqref{eq:coercivity_continuity} and~\eqref{eq:equivalence_preconditioner_energy_norm} yield 
  \begin{equation*}
	(\vec{P}_\ell\vec{B}_\ell\vec{x}_\ell, \vec{x}_\ell )_{\vec{P}^{-1}_\ell} = (\vec{B}_{\ell}\vec{x}_{\ell}, \vec{x}_{\ell})_2 = b(v_{\ell}, v_{\ell}) \eqreff{eq:coercivity_continuity}\geq C_{\textup{ell}} \, \enorm{v_\ell}^2.
  \end{equation*}
  In combination with Lemma~\ref{lem:equivalence_preconditioner_norm}, we further get that
  \begin{align*}
    1-\frac{(\vec{P}_\ell\vec{B}_\ell\vec{x}_\ell, \vec{x}_\ell )_{\vec{P}^{-1}_\ell}^2}{\norm{ \vec{P}_\ell \vec{B}_\ell\vec{x}_\ell }_{\vec{P}^{-1}_\ell}^2}  \eqreff{eq:equivalence_preconditioner_norm}\leq 1- \frac{C_{\textup{ell}}^2 \enorm{v_{\ell}}^4}{C_{\textup{bnd}}^2C_{\textup{pre}}^2\enorm{v_\ell}^2} \eqreff{eq:equivalence_preconditioner_energy_norm}\leq 1 - \frac{C_{\textup{ell}}^2}{C_{\textup{bnd}}^2C_{\textup{pre}}^4} \norm{\vec{x}_{\ell}}^2_{\vec{P}_{\ell}^{-1}}
  \end{align*}
  Altogether, we conclude that
\begin{align*}
    \norm{1-\alpha\vec{P}_\ell
	\vec{B}_\ell}_{\vec{P}^{-1}_\ell} 
  \leq \left(1-\frac{C_{\textup{ell}}^2}{C_{\textup{bnd}}^2} 
	\frac{1}{
    C_{\textup{pre}}^4}\right)^{1/2}
	\eqqcolon q_\mathrm{alg} < 1.
\end{align*}
This shows~\eqref{eq:contr_gmres} and hence concludes the proof.
\end{proof}

\begin{remark}
Note that Proposition~\ref{lem:contraction_gmres} ensures, in particular, discrete contraction for $k_{\max} = 1$. In this case only a one-dimensional Arnoldi basis is built in every step of PGMRES.
\end{remark}

\subsection{Multilevel additive Schwarz preconditioner}\label{sec:AS}

To define the multilevel additive Schwarz preconditioner
\(\vec{P}_{\ell}^{\AS} \in \R^{N_{\ell}\times N_{\ell}}\), we first introduce the necessary notation.
Let \(\mathcal{V}_\ell\) denote the set of all vertices of the triangulation
\(\TT_\ell\).
For each vertex \(z \in \mathcal{V}_\ell\), we define the associated
vertex patch by
\begin{equation}\label{eq:vertex_patch}
	\TT_{\ell}(z)
	\coloneqq 
	\{ T \in \TT_{\ell} \colon z \in T \} \quad \text{and} \quad \omega_{\ell}(z) \coloneqq \textnormal{interior} \big(\bigcup_{T \in \TT_{\ell}(z)} T \big).
\end{equation}
Denote by \(\mathcal{V}_\ell^{+}\) the set of newly created
vertices in \(\TT_\ell\) and their direct neighbors, i.e.,
\begin{equation}\label{eq:vertex_sets}
	\mathcal{V}_\ell^{+} 
	\coloneqq 
	\mathcal{V}_\ell \setminus \mathcal{V}_{\ell-1}
	\cup 
	\{ z' \in \mathcal{V}_{\ell} \cap \mathcal{V}_{\ell-1}
	\colon \omega_{\ell}(z') \neq \omega_{\ell-1}(z') \}.
\end{equation}
To define the correction subspaces underlying the preconditioner, we set
\begin{equation*}
\XX_{\ell}^q \coloneqq \SS_0^{q}(\TT_{\ell}), 
\quad 
\XX^q_{\ell,z} \coloneqq \SS_0^{q}(\TT_{\ell}(z)) \quad \text{for all } z \in \mathcal{V}_\ell, \text{ and } q \in \{1,p\}.
\end{equation*}
To build the preconditioner on level $\ell$ we employ $q=1$ for $0 \leq \ell' < \ell$ respectively $q=p$ for $\ell' = \ell$.
Let \(\vec{A}_{\ell} \in \R^{N_{\ell}\times N_{\ell}}\) denote the Galerkin matrix for principal part \(a(\cdot,\cdot)\), i.e.,
\begin{equation*}
(\vec{A}_{\ell})_{jk}
\coloneqq
a(\varphi_{\ell,k}, \varphi_{\ell,j}),
\quad \text{for all } j,k = 1,\dots,N_{\ell}.
\end{equation*}
Note that the induced norm satisfies $\norm{\vec{x}_{\ell}}_{\vec{A}_{\ell}} = \enorm{u_{\ell}}$ for $u_\ell = \sum_{j=1}^{N_{\ell}} (\vec{x})_j \varphi_{\ell,j}$ and all $\vec{x_{\ell}} \in \R^{N_{\ell}}$.
For $0 \leq \ell' < \ell$, let
\(\II^{1}_{\ell'} : \XX_{\ell'}^1 \to \XX^{p}_\ell = \XX_{\ell}\) denote the canonical embedding,
with matrix representation
\(\vec{I}^{1}_{\ell'} \in \R^{N_{\ell}\times N_{\ell'}^1}\),
where \(N_{\ell'}^1 \coloneqq \dim(\XX_{\ell'}^1)\).
Similarly, the embeddings
\(\II^{p}_{\ell,z} : \XX^{p}_{\ell,z} \to \XX^{p}_\ell = \XX_{\ell}\)
are represented by matrices
\(\vec{I}^{p}_{\ell,z} \in \R^{N_{\ell}\times N^{p}_{\ell,z}}\),
with \(N_{\ell,z} \coloneqq \dim(\XX^{p}_{\ell,z})\).

Denote by \(\vec{A}_{\ell'}^1\) and \(\vec{A}^{p}_{\ell,z}\) the Galerkin matrices
corresponding to the spaces \(\XX_{\ell'}^1\) and \(\XX^{p}_{\ell,z}\), respectively.
For each $0 \leq \ell' < \ell$, we define the diagonal matrix
\((\vec{D}_{\ell'}^+)^{-1} \in \R^{N_{\ell'}^1 \times N_{\ell'}^1}\) by
\begin{equation*}
[(\vec{D}^+_{\ell'})^{-1}]_{jk}
\coloneqq
\begin{cases}
\delta_{jk} \, ([\vec{A}^1_{\ell'}]_{jj})^{-1},
& \text{if } z_j \in \mathcal{V}_{\ell'}^+, \\
0,
& \text{otherwise}.
\end{cases}
\end{equation*}
With this notation, the additive Schwarz preconditioner is defined as
\begin{equation*}
\vec{P}_{\ell}^{\AS}
\coloneqq
\vec{I}^{1}_0 (\vec{A}^{1}_0)^{-1} (\vec{I}^{1}_0)^T
+ \sum_{\ell' = 1}^{\ell-1}
\vec{I}^{1}_{\ell'} (\vec{D}_{\ell'}^+)^{-1} (\vec{I}^{1}_{\ell'})^T
+ \sum_{z \in \mathcal{V}_{\ell}}
\vec{I}^{p}_{\ell,z} (\vec{A}^{p}_{\ell,z})^{-1} (\vec{I}^{p}_{\ell,z})^T.
\end{equation*}

\begin{remark}
Note that the matrices $\vec{I}_{\ell'}^1$ and $\vec{I}_{\ell,z}^p$ are introduced solely for the purpose of the analysis and are not used in the implementation.
In practice, only the prolongation and restriction matrices between consecutive levels are required. 
As is standard practice in multigrid methods, these can be assembled in $\mathcal{O}(\# \VV_{\ell'}^+)$ operations on the intermediate levels and in $\mathcal{O}(\# \TT_{\ell})$ operations on the finest level, thus preserving the overall linear complexity of the method; see, e.g.,~\cite{XQ94}.
\end{remark}

\begin{lemma}[robust additive Schwarz preconditioner]\label{lemma:additive_schwarz}
The additive Schwarz preconditioner \(\vec{P}_{\ell}^{\AS}\) is SPD and satisfies the norm
  equivalence~\eqref{eq:equivalence_preconditioner_energy_norm}, where the constant \(C_{\textup{pre}}\) depends only on $\Omega$, $\TT_0$, $d$, $\gamma$-shape regularity, and~$\vec{K}$.
\end{lemma}

\begin{proof}
From~\cite[Corollary 19]{hmp26} and~\cite[Theorem C.1]{tw05}, it follows that \(\vec{P}_{\ell}^{\AS}\) is SPD and satisfies the spectral equivalence
\begin{equation}\label{eq:norm_equivalence_AS}
   c
   \le
   \lambda_{\min}(\vec{P}_{\ell}^{\AS} \vec{A}_{\ell})
   =
   \norm{ (\vec{P}_{\ell}^{\AS}\vec{A}_{\ell})^{-1} }^{-1}_{\vec{A}_{\ell}}
   \quad \text{and} \quad
   \norm{\vec{P}_{\ell}^{\AS}\vec{A}_{\ell}}_{\vec{A}_{\ell}}
   =
   \lambda_{\max}(\vec{P}_{\ell}^{\AS}\vec{A}_{\ell})
   \le
   C,
\end{equation}
where the constants \(c\) and \(C\) depend only on $\Omega$, $\TT_0$, $d$, $\gamma$-shape regularity, and $\vec{K}$.
Let
\(v_\ell = \sum_{j=1}^{N_{\ell}} (\vec{x}_\ell)_j \varphi_{\ell,j} \in \XX_{\ell}\).
Using~\eqref{eq:norm_equivalence_AS}, we obtain
\begin{equation*}
\norm{ \vec{x}_{\ell} }_{(\vec{P}_{\ell}^{\AS})^{-1}}^2
=
((\vec{P}_{\ell}^{\AS}\vec{A}_{\ell})^{-1} \vec{x}_{\ell}, \vec{x}_{\ell})_{\vec{A}_{\ell}}
\le
\norm{ (\vec{P}_{\ell}^{\AS}\vec{A}_{\ell})^{-1}}_{\vec{A}_{\ell}}
\norm{\vec{x}_{\ell}}_{\vec{A}_{\ell}}^2
\eqreff{eq:norm_equivalence_AS}{\leq}
c^{-1} \enorm{v_{\ell}}^2.
\end{equation*}
Moreover,
\begin{equation*}
(\vec{P}_{\ell}^{\AS}\vec{A}_{\ell}\vec{x}_{\ell}, \vec{x}_{\ell})_{\vec{A}_{\ell}}
\le
\norm{\vec{P}_{\ell}^{\AS}\vec{A}_{\ell}}_{\vec{A}_{\ell}}
\norm{\vec{x}_{\ell}}_{\vec{A}_{\ell}}^2
\eqreff{eq:norm_equivalence_AS}{\leq}
C\,\enorm{v_{\ell}}^2.
\end{equation*}
Setting \(\vec{y}_{\ell} \coloneqq \vec{P}_{\ell}^{\AS}\vec{A}_{\ell}\vec{x}_{\ell}\), we obtain that
\begin{equation*}
\norm{ \vec{x}_{\ell} }_{(\vec{P}_{\ell}^{\AS})^{-1}}
=
\sup_{\vec{y}_{\ell} \neq 0}
\frac{(\vec{x}_{\ell}, \vec{y}_{\ell})_{(\vec{P}_{\ell}^{\AS})^{-1}}}
     {\|\vec{y}_{\ell}\|_{(\vec{P}_{\ell}^{\AS})^{-1}}}
\ge
\frac{(\vec{x}_{\ell}, \vec{x}_{\ell})_{\vec{A}_{\ell}}}
     {((\vec{P}_{\ell}^{\AS})^{-1}\vec{A}_{\ell}\vec{x}_{\ell},
       \vec{x}_{\ell})_{\vec{A}_{\ell}}^{1/2}}
\ge
C^{-1/2}\,\enorm{v_{\ell}}.
\end{equation*}
This concludes the proof with $C_{\textup{pre}}= \max\{\sqrt{C}, 1 / \sqrt{c}\}$.
\end{proof}

\subsection{Symmetric multigrid preconditioner}

As an alternative, we consider the symmetric multigrid preconditioner from~\cite[Section 5]{hmp26} as symmetrized and linearized version of~\cite{imps2022}.
In the following, we abbreviate $\prod_{i=0}^{\ell} \vec{C}_i \coloneqq \vec{C}_0 \vec{C}_1 \cdots \vec{C}_{\ell}$ for any matrices $\vec{C}_0, \dots, \vec{C}_{\ell}$.
Let \(\mu \coloneqq 1/(d+1)\) and define, for \(\ell' \in \{1,\dots,\ell\}\),
\begin{equation*}
\vec{M}^{\downarrow}_{\ell'}
\coloneqq
\vec{I}^{1}_{\ell'}(\vec{D}_{\ell'}^+)^{-1}(\vec{I}^{1}_{\ell'})^{T}
\prod_{i=\ell'+1}^{\ell}
(\vec{I}-\mu \vec{A}_{\ell}\vec{I}^{1}_{i}(\vec{D}_{i}^+)^{-1}(\vec{I}_i^{1})^{T}),
\end{equation*}
for $\ell' = 0$,
\begin{equation*}
	\vec{M}^{\uparrow}_{0}
\coloneqq
\vec{I}^{1}_{0}(\vec{A}_0^1)^{-1}(\vec{I}^{1}_{0})^{T}
\prod_{i=1}^{\ell}
(\vec{I}-\mu \vec{A}_{\ell}\vec{I}^{1}_{i}(\vec{D}_{i}^+)^{-1}(\vec{I}_i^{1})^{T}),
\end{equation*}
and, for \(\ell' \in \{1,\dots,\ell\}\),
\begin{equation*}
\vec{M}^{\uparrow}_{\ell'}
\coloneqq
\vec{I}_{\ell'}^{1} (\vec{D}_{\ell'}^+)^{-1} (\vec{I}^{1}_{\ell'})^{T}
\prod_{i=1}^{\ell'-1}
(\vec{I}-\mu \vec{A}_{\ell}\vec{I}^{1}_{\ell'-i}(\vec{D}_{\ell' -i}^+)^{-1} (\vec{I}_{\ell' - i}^{1})^{T})
\prod_{i=0}^{\ell}
(\vec{I}-\mu \vec{A}_{\ell}\vec{I}^{1}_{i}(\vec{D}_{i}^+)^{-1}(\vec{I}^{1}_i)^{T}).
\end{equation*}
For notational convenience, we set here \(\vec{D}_0^+ \coloneqq \mu \vec{A}_0^1\), $(\vec{D}^+_{\ell})^{-1} \coloneqq \sum_{z \in \mathcal{V}_{\ell}} \vec{I}^{p}_{\ell,z} (\vec{A}^{p}_{\ell,z})^{-1}(\vec{I}^{p}_{\ell,z})^T$, and $\vec{I}^{1}_{\ell} \coloneqq \vec{I}$.
The symmetric multigrid preconditioner is then defined as
\begin{equation}\label{eq:sMG_preconditioner}
\vec{P}_{\ell}^{\SMG} \coloneqq
\mu \sum_{\ell'=1}^{\ell} \vec{M}^{\downarrow}_{\ell'} + \vec{M}^{\uparrow}_{0} + \mu \sum_{\ell'=1}^{\ell} \vec{M}^{\uparrow}_{\ell'}.
\end{equation}

\begin{lemma}[symmetric multigrid preconditioner]
The symmetric multigrid preconditioner \(\vec{P}_{\ell}^{\SMG}\) is SPD and satisfies~\eqref{eq:equivalence_preconditioner_energy_norm} with constant
\(C_{\textup{pre}}= \max\{\sqrt{1+q},\, 1/\sqrt{1-q}\}\), where $0 < q < 1$ depends only on $\Omega$, $\TT_0$, $d$, $\gamma$-shape regularity, and $\vec{K}$.
\end{lemma}

\begin{proof}
	From the proof of~\cite[Corollary 16]{hmp26}, we obtain estimates analogous
to~\eqref{eq:norm_equivalence_AS}.
In particular, there exists a constant \(0 < q < 1\), independent of \(h\) and
\(p\), such that
\begin{equation*}
\norm{(\vec{P}_{\ell}^{\SMG}\vec{A}_{\ell})^{-1}}_{\vec{A}_{\ell}}
\le
\frac{1}{1-q},
\qquad
\norm{\vec{P}_{\ell}^{\SMG}\vec{A}_{\ell}}_{\vec{A}_{\ell}}
\le
1+q.
\end{equation*}
Moreover,~\cite[Lemma 15]{hmp26} states that \(\vec{P}_{\ell}^{\SMG}\) is SPD.
With the arguments from the proof of Lemma~\ref{lemma:additive_schwarz}, this concludes the proof.
\end{proof}

\section{AFEM with a-posteriori-steered parameter control} \label{section:afem_algorithm}

In this section, we present a new adaptive algorithm that employs PGMRES from Algorithm~\ref{algorithm:unconditional_afem_gmres} as its algebraic solver.
First, we recall a result from~\cite{cfpp2014} establishing the equivalence between tail summability, inverse summability, and R-linear convergence for any nonnegative sequence of scalars.
This result is central both to the design of the algorithm and to the proof of unconditional full R-linear convergence in Section~\ref{section:unconditional_full_R_linear_convergence}.

\begin{lemma}[summability and R-linear convergence]\label{lemma:summability}
Let $(a_{\ell})_{\ell \in \N_0}$ be a sequence in $\R_{\geq 0}$ and $\beta, \gamma > 0$.
Then, the following statements~\eqref{lemma:i:tail_summaility}--\eqref{lemma:iii:linear_conv} are pairwise equivalent
\begin{enumerate}[label=(\roman*), ref = \textnormal{\roman*}, font = \upshape]
	\item \label{lemma:i:tail_summaility} \textbf{tail summability:} There exists a constant $C_{\beta} > 0$, such that
	\begin{equation}\label{eq:tail_summability_one}
		\sum_{\ell' = \ell+1}^{\infty} a^{\beta}_{\ell'} \leq C_{\beta} \, a^{\beta}_{\ell} \quad \text{for all } \ell \in \N_0.
	\end{equation}
	\item \label{lemma:ii:inverse_tail_summability} \textbf{inverse summability:} There exists a constant $C_{\gamma} > 0$, such that 
	\begin{equation}\label{eq:inverse_summability}
		\sum_{\ell' = 0}^{\ell-1} a^{-\gamma}_{\ell'} \leq C_{\gamma} \, a^{-\gamma}_{\ell} \quad \text{for all } \ell \in \N_0.
	\end{equation}
	\item \label{lemma:iii:linear_conv} \textbf{R-linear convergence:} There exist constants $\widetilde{C}_{\textup{lin}} > 1$ and $0 < \widetilde{q}_{\textup{lin}} <1$ with
\begin{equation}\label{eq:R_linear_conv}
    a_{\ell} \leq \widetilde{C}_{\textup{lin}} \, \widetilde{q}_{\textup{lin}}^{\ell-\ell'} a_{\ell'} \quad \text{for all } 0 \leq \ell' \leq \ell.
\end{equation}
\end{enumerate}
Let $(a_{\ell})_{\ell \in \N_0}$  have the additional property
\begin{equation}\label{eq:add_prop_summability}
	\forall \ell \in \N_0 : \Big(a_{\ell} = 0 \; \Longrightarrow \; a_{\ell'}=0 \quad \text{for all } \ell' > \ell\Big).
\end{equation}
Let $\ell_0 \in \N$.
If any of the statements~\eqref{lemma:i:tail_summaility}--\eqref{lemma:iii:linear_conv} hold only for $\ell \geq \ell_0$, then they already hold for all $\ell \in \N_0$ with potentially larger constants depending also on $(a_{\ell'})_{\ell'=0}^{\ell_0}$.
\end{lemma}

\begin{proof}
The equivalence of statements~\eqref{lemma:i:tail_summaility}--\eqref{lemma:iii:linear_conv} was proved in~\cite[Lemma~4.9]{cfpp2014}, albeit without the exponent in the tail summability condition; see also~\cite[Lemma~2]{bfmps2025}.
Let \((a_{\ell})_{\ell \in \N_0}\) satisfy~\eqref{eq:add_prop_summability}. We show that tail summability~\eqref{eq:tail_summability_one} for all \(\ell \ge \ell_0\) implies tail summability~\eqref{eq:tail_summability_one} for all \(\ell \in \N_0\). The corresponding implications for~\eqref{lemma:ii:inverse_tail_summability} and~\eqref{lemma:iii:linear_conv} follow similarly.
To this end, let \(\ell < \ell_0\) and define \(N \coloneqq \set{0 \le \ell' < \ell_0 \colon a_{\ell'} \neq 0}.\)
Moreover, set
\begin{equation*}
    \widetilde{C}_m \coloneqq \frac{C_m a_{\ell_0}^m + \sum_{\ell'=1}^{\ell_0} a_{\ell'}^m}{\min_{\ell' \in N} a_{\ell'}^m}.
\end{equation*}
Then, using~\eqref{eq:tail_summability_one} for \(\ell_0\) and the assumption~\eqref{eq:add_prop_summability} in the case $a_{\ell}=0$, we obtain
\begin{equation*}
    \sum_{\ell'=\ell+1}^{\infty} a_{\ell'}^m = \sum_{\ell'=\ell_0+1}^{\infty} a_{\ell'}^m + \sum_{\ell'=\ell+1}^{\ell_0} a_{\ell'}^m 
    \eqreff{eq:tail_summability_one}\le C_m a_{\ell_0}^m + \sum_{\ell'=\ell+1}^{\ell_0} a_{\ell'}^m 
    \eqreff{eq:add_prop_summability}\le \widetilde{C}_m \, a_{\ell}^m.
\end{equation*}
This concludes the proof.
\end{proof}

Given an approximation 
\( u_{\ell}  \in \XX_{\ell} \) of the solution \( u_{\ell} ^\star \in \XX_{\ell}  \) 
to \eqref{eq:discrete_formulation} and its vector representation $u_{\ell} = \sum_{j=1}^{N_{\ell}} (\vec{x}_{\ell})_j \varphi_{\ell,j}$, Algorithm~\ref{alg:optimally_preconditioned_gmres} returns an improved 
approximation $\vec{y}_{\ell} \in \R^{N_{\ell}}$.
This gives rise to the map
\(\Psi_{\ell} \colon \XX_{\ell}  \to \XX_{\ell} \), $\Psi_{\ell}(u_{\ell}) \coloneqq \sum_{j=1}^{N_{\ell}} (\vec{y}_{\ell})_j \, \varphi_{\ell,j}$.
In the following algorithm, the iterates satisfy $u_{\ell}^k = \sum_{j=1}^{N_{\ell}} (\vec{x}_{\ell}^k)_j \varphi_{\ell,j}$ and $\vec{s}_{\ell}^k = \vec{P}_{\ell} (\vec{d}_{\ell}-\vec{B}_{\ell}\vec{x}_{\ell}^k)$, where $\vec{x}_{\ell}^k$ are the iterates generated by Algorithm~\ref{alg:optimally_preconditioned_gmres} with initial guess $\vec{x}_{\ell}^0$. 

\begin{algorithm}[unconditional AFEM with optimal PGMRES solver] \label{algorithm:unconditional_afem_gmres}
\textbf{Input:} Initial mesh $\mathcal{T}_0$, adaptivity parameters $0 < \theta \leq 1$ and $\Cmark \geq 1$, initial guess $u_0^0 \in \XX_0$, as well as $\Calg, \lalg > 0$, ${\sf S} \coloneqq 0$, $k_{\max}$, and  tolerance $\tau \geq 0$.  \\
Perform the following steps~\eqref{alg2:single_i}--\eqref{alg2:single_iv} for all $\ell=0,1,2,\dots$:
\begin{enumerate}[label=(\roman*), ref = \textnormal{\roman*}, font = \upshape]
  \item\label{alg2:single_i} \textbf{Solve \& Estimate:} Compute $u_\ell^{\underline{k}[\ell]} \coloneqq \Psi_\ell(u_\ell^{0})$ via the optimally preconditioned GMRES method from Algorithm~\ref{alg:optimally_preconditioned_gmres} using \(\Lambda (\boldsymbol{x}_\ell^{k}) \coloneqq \lalg \, \eta_\ell(u_\ell^{k})\) and $k_{\max}$ for $k \leq \underline{k}[\ell]$, where $\underline{k}[\ell] \in \N$ denotes the number of PGMRES iterations on level $\ell$.
  \item\label{alg2:single:ii} \textbf{Adaptive Parameter Control:} If $\eta_{\ell}(u_{\ell}^{\underline{k}})+ \norm{\vec{s}_{\ell}^{\underline{k}}}_{\vec{P}^{-1}_{\ell}} \leq \tau$, then stop Algorithm~\ref{algorithm:unconditional_afem_gmres} and terminate; else do:
  \begin{enumerate}[label=(\alph*), ref = \textnormal{ii.\alph*}, font = \upshape]
    \item \label{alg2:ii:a} If $\ell \geq 1$ and ${\sf S} > \Calg (\eta_{\ell}(u_{\ell}^{\underline{k}})+ \norm{\vec{s}_{\ell}^{\underline{k}}}_{\vec{P}^{-1}_{\ell}})^{-1}$, then $\Calg \mapsfrom 2 \Calg$, $\lalg \mapsfrom \lalg/2$.
    \item \label{alg2:ii:b} Always update ${\sf S} \mapsfrom {\sf S} + (\eta_{\ell}(u_{\ell}^{\underline{k}})+ \norm{\vec{s}_{\ell}^{\underline{k}}}_{\vec{P}^{-1}_{\ell}})^{-1}$.
  \end{enumerate}
  \item \textbf{Mark:} Determine a set $\mathcal{M}_\ell \in \mathbb{M}_\ell[\theta, u_\ell^{\underline{k}}] \coloneqq 
	\{\mathcal{U}_\ell \subseteq \mathcal{T}_\ell
	\,:\, \theta \eta_\ell(u_\ell^{\underline{k}})^2 
	\le \eta_\ell(\mathcal{U}_\ell,
	u_\ell^{\underline{k}})^2\}$ satisfying the D\"orfler marking criterion 
	\begin{equation}\label{eq:doerfler}
	\#\mathcal{M}_\ell \le C_{\textup{mark}} \!\! 
	\min_{\mathcal{U}_\ell \in 
	\mathbb{M}_\ell[\theta,u_\ell^{\underline{k}}]}\!\! 
	\#\mathcal{U}_\ell.
	\end{equation}
  \item\label{alg2:single_iv} \textbf{Refine:} Generate $\mathcal{T}_{\ell+1} \coloneqq \mathtt{refine}(\mathcal{T}_\ell,\mathcal{M}_\ell)$ and employ nested iteration $u_{\ell+1}^0 \coloneqq u_\ell^{\underline{k}}$.
\end{enumerate}
\textbf{Output:} Sequence of meshes $\{\mathcal{T}_\ell\}_{\ell \in \N_0}$ and discrete solutions
$\{u_\ell^{\underline{k}}\}_{\ell \in \N_0}$.
\end{algorithm}

Note that Algorithm~\ref{algorithm:unconditional_afem_gmres} consists of two nested loops, where we denote the discretization with the index $\ell$ and the iterations of the algebraic solver, which in this case is PGMRES, with $k$. 
More precisely, the function $u_{\ell}^k$ for $k \leq \underline{k}[\ell]$ denotes the $k$-th approximation of $u_{\ell}^{\star}$ calculated by Algorithm~\ref{alg:optimally_preconditioned_gmres} on the discretization level $\ell$ with initial guess $u_{\ell}^0$.
The sequential nature of Algorithm~\ref{algorithm:unconditional_afem_gmres} 
results in the countably infinite index set
\begin{equation*}
	\mathcal{Q} \coloneqq \bigl\{(\ell,k) \in \mathbb{N}_0^2 
	\colon u_\ell^k \in \mathcal{X}_\ell 
	\text{ appears in Algorithm~\ref{algorithm:unconditional_afem_gmres}} \bigr\}
\end{equation*}
equipped with the lexicographic ordering
\begin{equation*}
	(\ell',k') \le (\ell,k)
	\quad :\Longleftrightarrow \quad
	\text{$u_{\ell'}^{k'}$ is defined not later than $u_\ell^k$ in
		Algorithm~\ref{algorithm:unconditional_afem_gmres}}
\end{equation*}
and the total step counter
\begin{equation*}
	| \ell, k | \coloneqq \# \bigl \{(\ell', k') 
	\in \mathcal{Q} \,:\, (\ell', k') \le
		(\ell, k) \bigr \} \in \mathbb{N}_0
	\quad \text{for all \((\ell, k) \in \mathcal{Q}\).}
\end{equation*}
Moreover, we define the stopping indices
\begin{equation}
	\begin{aligned}
		\underline{\ell}    
		& \coloneqq \sup \{\ell \in \mathbb{N}_0 \,:\, 
		(\ell,0) \in \mathcal{Q}\} \in \mathbb{N}_0 \cup \{\infty\},
		\\
		\underline{k}[\ell] & \coloneqq \sup \{k \in \mathbb{N}_0 \,:\, 
		(\ell,k) \in \mathcal{Q}\} \in \mathbb{N} \cup \{\infty\},
		\quad \text{whenever } (\ell,0) \in \mathcal{Q}.
	\end{aligned}
\end{equation}
We stress that the definition of $k[\ell]$ is consistent with that of 
Algorithm~\ref{algorithm:unconditional_afem_gmres}\eqref{alg2:single:ii}.

\begin{remark}[adaptive parameter control]\label{remark:afem_plus_gmres}
Let us define the computable quasi-error
\begin{equation}\label{eq:computable_quasierror}
	{\sf H}_\ell^k \coloneqq
	\eta_\ell(u_\ell^k) + \norm{\vec{s}^k_{\ell}}_{\boldsymbol{P}^{-1}_\ell}
	\quad \text{for all} \quad (\ell,k) \in \QQ
\end{equation}
and note that ${\sf S} = \sum_{j=0}^\ell ({\sf H}_j^{\underline{k}})^{-1}$ in Algorithm~\ref{algorithm:unconditional_afem_gmres}\eqref{alg2:ii:b}.
Comparing Algorithm~\ref{algorithm:unconditional_afem_gmres} to standard adaptive algorithms (see, e.g.,~\cite[Algorithm~B]{bfmps2025}), we observe that the difference lies in the adaptive parameter control performed in step~\eqref{alg2:single:ii}.
This new a-posteriori-steered criterion aims to algorithmically guarantee R-linear convergence~\eqref{eq:R_linear_conv} of ${\sf H}_{\ell}^{\underline{k}}$ across the levels $\ell$ by testing the equivalent inverse summability condition~\eqref{eq:inverse_summability} for $\gamma=1$.
\end{remark} 

The following lemma provides a-posteriori error control of the total error.

\begin{lemma}[a-posteriori error control]\label{lemma:aposteriori}
Suppose that the estimator satisfies
	\eqref{axiom:stability}--\eqref{axiom:discrete_reliability} and that the employed preconditioner satisfies~\eqref{eq:equivalence_preconditioner_energy_norm}.
    Recall the computable quasi-error ${\sf H}_{\ell}^k$ from~\eqref{eq:computable_quasierror}.
With the constants $\Cstab, \Cell,\Cbnd$, and $C_{\textup{pre}}$ from~\eqref{axiom:stability},~\eqref{eq:coercivity_continuity}, and~\eqref{eq:equivalence_preconditioner_energy_norm}, define \(C_{\textup{eq}} \coloneqq \max \{1, (1+ C_{\textup{stab}}) \,
    C_{\textup{ell}}^{-1} \, C_{\textup{pre}}, (C_{\textup{bnd}} \, C_{\textup{pre}}+ C_{\textup{stab}})  \} > 0\).
		Then, there holds 
\begin{equation}\label{eq:apost_quasierror}
	C^{-1}_{\textup{eq}} \big( \eta_\ell(u_\ell^\star) + \enorm{u_\ell^\star - u_\ell^k} \big)
	\le
	{\sf H}_\ell^k 
	\le
	C_{\textup{eq}} \big( \eta_\ell(u_\ell^\star) + \enorm{u_\ell^\star - u_\ell^k} \big)
	\quad \text{for all} \ (\ell,k) \in \mathcal{Q}.
\end{equation}
With \( C'_{\textup{rel}} \coloneqq C_{\textup{eq}} \max\{C_{\textup{rel}},1\} >0\), this yields
	\begin{equation}\label{eq:apost_totalerror}
	\enorm{u^\star - u_\ell^k} 
	\le
    C'_{\textup{rel}} \, {\sf H}_\ell^k, \quad \text{for all} \ (\ell,k) \in \mathcal{Q}.
	\end{equation}
\end{lemma}

\begin{proof}
We start with the lower bound 
in~\eqref{eq:apost_quasierror}. Stability~\eqref{axiom:stability} and equivalence~\eqref{eq:equivalence_residual_norm} yield
\begin{align*}
		\eta_\ell(u_\ell^\star) + \enorm{u_\ell^\star - u_\ell^k}
		&\eqreff*{axiom:stability}\le
		\eta_\ell(u_\ell^k) + (1+ C_{\textup{stab}}) \,
		\enorm{u_\ell^\star - u_\ell^k} \\
		&\eqreff*{eq:equivalence_residual_norm}\le
		\eta_\ell(u_\ell^k) + (1+ C_{\textup{stab}}) \,
    C_{\textup{ell}}^{-1} \, C_{\textup{pre}} \,
		\norm{r^k_{\ell}}_{\boldsymbol{P}^{-1}_\ell} 
		\le C_{\textup{eq}}  {\sf H}_\ell^k.
\end{align*}
Similarly, for the upper bound in \eqref{eq:apost_quasierror},
equivalence~\eqref{eq:equivalence_residual_norm} and
stability~\eqref{axiom:stability} imply
\begin{align*}
		{\sf H}_\ell^k
		=
		\eta_\ell(u_\ell^k) &+ \norm{r^k_{\ell}}_{\boldsymbol{P}^{-1}_\ell} 
		\eqreff*{eq:equivalence_residual_norm}\le
		\eta_\ell(u_\ell^k) + 
    C_{\textup{bnd}} \, C_{\textup{pre}}\,
		\enorm{u_\ell^\star - u_\ell^k} \\
		&\eqreff*{axiom:stability}\le
    \eta_\ell(u_\ell^\star) + (C_{\textup{bnd}} \, C_{\textup{pre}}+ C_{\textup{stab}}) \,
		\enorm{u_\ell^\star - u_\ell^k} 
		\le  C_{\textup{eq}} \big(\eta_\ell(u_\ell^\star) + \enorm{u_\ell^\star - u_\ell^k} \big).
\end{align*}
To show \eqref{eq:apost_totalerror}, we use reliability~\eqref{axiom:reliability} and the established equivalence~\eqref{eq:apost_quasierror} to see that
    \begin{equation}\label{eq:error_quasi_error_bound}
	\enorm{u^\star- u_\ell^k} 
	\le
    \enorm{u^\star - u_\ell^\star} +
    \enorm{u_\ell^\star - u_\ell^k} 
	\eqreff{axiom:reliability}\le
    C_{\textup{rel}} \, \eta_\ell(u_\ell^\star) 
	+ \enorm{u_\ell^\star - u_\ell^k} 
	\eqreff{eq:apost_quasierror}\le C'_{\textup{rel}} \,{\sf H}_\ell^k. 
	\end{equation}
	This concludes the proof.
\end{proof}


\section{Unconditional full R-linear convergence of Algorithm~\ref{algorithm:unconditional_afem_gmres}} \label{section:unconditional_full_R_linear_convergence}

To prove unconditional full R-linear convergence of Algorithm~\ref{algorithm:unconditional_afem_gmres}, we first show conditional tail summability of ${\sf H}_{\ell}^{\underline{k}}$ in Lemma~\ref{theorem:full_linear_convergence}.
This combined with the new adaptive parameter criterion of Algorithm~\ref{algorithm:unconditional_afem_gmres}\eqref{alg2:ii:a} implies that $\Calg$ and $\lalg$ are only updated finitely often; see Lemma~\ref{lemma:bounded_updates}.
Then, Theorem~\ref{theorem:unconditional_full_R_linear_convergence} uses this result to derive unconditional full R-linear convergence of the quasi-error ${\sf H}_{\ell}^k$, i.e., with no additional assumptions on the input parameters $\theta, \Calg, \lalg$, and $k_{\max}$.
The following lemma establishes estimator equivalence in the spirit of~\cite[Lemma 4.9]{ghps2018} that enables to transfer the Dörfler criterion from the exact discrete solution to the inexact discrete solution, and vice-versa.

\begin{lemma}[estimator equivalence]\label{lem:estimator_equivalence}
	Suppose that the estimator satisfies \eqref{axiom:stability} and that the employed preconditioner satisfies~\eqref{eq:equivalence_preconditioner_energy_norm}.
  Let $\lalg^\star \coloneqq \min\{1, \Cstab^{-1} C_{\textup{solve}}^{-1}\}$ with $C_{\textup{solve}} \coloneqq C_{\textup{ell}}^{-1} \, C_{\textup{pre}}$.
	Then, for all
	\(0 < \theta \le 1\) and all
	\(
		0
		<
		\lalg
		<
		\lalg^\star
	\), 
	there holds the equivalence
	\begin{equation}\label{eq:estimator_equivalence}
		\bigl[
			1- \lalg/\lalg^\star
		\bigr] \,
		\eta_{\ell}(u_\ell^{\underline{k} })
		\le
		\eta_{\ell}(u_\ell^\star)
		\le
		\bigl[
			1 + \lalg/\lalg^\star
		\bigr] \,
		\eta_{\ell}(u_\ell^{\underline{k} }).
	\end{equation}
	Let 
	\begin{equation}\label{eq:thetamark}
	\theta' \coloneqq \frac{(\theta^{1/2}-\lalg/\lalg^\star)^2}{(1+\lalg/\lalg^\star)^2} \quad \text{and} \quad \theta_{\textup{mark}} \coloneqq \frac{(\theta^{1/2}+ \lalg/\lalg^\star)^2}{(1 - \lalg/\lalg^\star)^2}.
	\end{equation}
	Then, Dörfler marking for $(u_{\ell}^{\underline{k}},\theta)$ implies Dörfler marking for $(u_{\ell}^\star, \theta')$, i.e., for any $\mathcal{R}_{\ell} \subseteq \TT_{\ell}$, there holds the following implication:
	\begin{equation}\label{eq:equivalence_Doerfler_reverse}
		\theta \, \eta_{\ell}(u_\ell^{\underline{k} })^2
		\le
		\eta_{\ell}(\mathcal{R}_{\ell}; u_\ell^{\underline{k} })^2
		\quad
		\Longrightarrow
		\quad
		\theta' \, \eta_{\ell}(u_\ell^\star)^2
		\le
		\eta_{\ell}(\mathcal{R}_{\ell}; u_\ell^\star)^2.
	\end{equation}
	Conversely, for any \(\mathcal{R}_{\ell} \subseteq \TT_{\ell}\), there also holds the 
	following implication:
	\begin{equation}\label{eq:equivalence_Doerfler}
		\theta_{\textup{mark}} \, \eta_{\ell}(u_\ell^\star)^2
		\le
		\eta_{\ell}(\mathcal{R}_{\ell}; u_\ell^\star)^2
		\quad
		\Longrightarrow
		\quad
		\theta \, \eta_{\ell}(u_\ell^{\underline{k} })^2
		\le
		\eta_{\ell}(\mathcal{R}_{\ell}; u_\ell^{\underline{k} })^2.
	\end{equation}
\end{lemma}

\begin{proof}
	Norm equivalence~\eqref{eq:equivalence_residual_norm} and the stopping
	criterion~\eqref{eq:single:termination} allow to estimate
	\begin{align}
		\label{eq1:proof:estimator_equivalence}
		\begin{split}
			\enorm{u_\ell^\star& - u_\ell^{\underline{k} }}
			\eqreff*{eq:equivalence_residual_norm}
			\le 
            C_{\textup{ell}}^{-1} \, C_{\textup{pre}} \,
                \norm{\vec{s}^{\underline{k}}_\ell}_{\boldsymbol{P}^{-1}_\ell} 
		\eqreff*{eq:single:termination}
			\le
            C_{\textup{ell}}^{-1} \, C_{\textup{pre}} \,
			\lalg
			\,
			\eta_{\ell}(u_\ell^{\underline{k} })
			=
			C_{\textup{solve}} \,
			\lalg \, \eta_{\ell}(u_\ell^{\underline{k} }).
		\end{split}
	\end{align}
	For any \(\mathcal{U}_\ell \subseteq \TT_\ell\), 
	stability~\eqref{axiom:stability}
	then allows to obtain
	\begin{align}\label{eq3:proof:estimator_equivalence}
		\begin{split}
		\eta_{\ell}(\mathcal{U}_\ell; u_\ell^{\underline{k} })
		\eqreff*{axiom:stability}
		\le
		\eta_{\ell}(\mathcal{U}_\ell; u_\ell^\star)
		+ \Cstab \, \enorm{u_\ell^\star - u_\ell^{\underline{k} }}
		\, &\eqreff*{eq1:proof:estimator_equivalence}\le \,
		\eta_{\ell}(\mathcal{U}_\ell; u_\ell^\star)
		+ \Cstab \, C_{\textup{solve}} \, \lalg \,
		\eta_{\ell}(u_\ell^{\underline{k} })
		\\
		&\le \,
		\eta_{\ell}(\mathcal{U}_\ell; u_\ell^\star)
		+
		\bigl(\lalg / \lalg^\star\bigr) \,
		\eta_{\ell}(u_\ell^{\underline{k} }).
		\end{split}
\end{align}
	Then, the lower bound
	of~\eqref{eq:estimator_equivalence} follows by taking
	$\mathcal{U}_\ell = \TT_{\ell}$, $0 < \lalg < \lalg^\star$,
	and rearranging the latter estimate. 
	Analogously, it follows
	\begin{align*}
		\medmuskip = 2mu
		\begin{split}
			\eta_{\ell}(\mathcal{U}_\ell; u_\ell^\star)
			\eqreff*{axiom:stability}
			\le
			\eta_{\ell}(\mathcal{U}_\ell; u_\ell^{\underline{k} })
			+ \Cstab \, \enorm{u_\ell^\star - u_\ell^{\underline{k} }}
			 \le \
			\eta_{\ell}(\mathcal{U}_\ell; u_\ell^{\underline{k} })
			+
			\bigl(\lalg / \lalg^\star \bigr) \,
			\eta_{\ell}(u_\ell^{\underline{k} }).
		\end{split}
	\end{align*}
	For \(\mathcal{U}_\ell = \TT_{\ell}\), this concludes the proof
	of~\eqref{eq:estimator_equivalence}. 
	Suppose that
	\(\mathcal{R}_{\ell} \subseteq \TT_\ell\)
	satisfies
	\(
		\theta^{1/2} \, \eta_{\ell}(u_\ell^{\underline{k}})
		\le
		\eta_{\ell}(\mathcal{R}_{\ell}; u_\ell^{\underline{k}})
	\).
Using~\eqref{eq3:proof:estimator_equivalence} for $\RR_{\ell}$, we obtain
	\begin{equation*}
		\theta^{1/2} \, \eta_{\ell}(u_\ell^{\underline{k}})
		\le
		\eta_{\ell}(\mathcal{R}_{\ell}; u_\ell^{\underline{k}}) \eqreff{eq3:proof:estimator_equivalence}\le \eta_{\ell}(\mathcal{R}_{\ell}; u_\ell^\star) + (\lalg / \lalg^\star) \,\eta_{\ell}(u_\ell^{\underline{k}}).
	\end{equation*}
	Together with~\eqref{eq:estimator_equivalence}, this yields
	\begin{align*}
		(\theta')^{1/2} \,\eta_{\ell}( u_\ell^{\star}) = \frac{\theta^{1/2}-\lalg/\lalg^\star}{1+ \lalg/\lalg^\star} \,
		\eta_{\ell}( u_\ell^{\star})
		\eqreff{eq:estimator_equivalence}
		\le
		[\theta^{1/2}- \lalg/\lalg^\star] \, \eta_{\ell}(u_\ell^{\underline{k}})
		\leq
		\eta_{\ell}(\mathcal{R}_{\ell}; u_\ell^{\star}).
	\end{align*}
	For the reverse implication, we assume $\theta_{\textup{mark}}^{1/2}\eta_{\ell}(u_\ell^{\star}) \le \eta_{\ell}(\mathcal{R}_{\ell}; u_\ell^{\star})$ and similarly obtain
	\begin{align*}
		\bigl[
			1
			&- \lalg / \lalg^\star
		\bigr] \,
		\theta_{\textup{mark}}^{1/2}
		\eta_{\ell}( u_\ell^{\underline{k} })
		\eqreff{eq:estimator_equivalence}
		\le
		\theta_{\textup{mark}}^{1/2} \, \eta_{\ell}(u_\ell^\star)
		\le
		\eta_{\ell}(\mathcal{R}_{\ell}; u_\ell^\star)
		\\
		&
		\le
		\eta_{\ell}(\mathcal{R}_{\ell}; u_\ell^{\underline{k} })
		+ (\lalg / \lalg^\star) \,
		\eta_{\ell}(u_\ell^{\underline{k} })
		\eqreff{eq:thetamark}
		=
		\eta_{\ell}(\mathcal{R}_{\ell}; u_\ell^{\underline{k} })
		+
		\Big(
			\theta_{\textup{mark}}^{1/2} \,
			\bigl[
				1
				- \lalg / \lalg^\star
			\bigr]
			- \theta^{1/2}
		\Big) \,
		\eta_{\ell}(u_\ell^{\underline{k} }).
	\end{align*}
	This yields
	\(
		\theta \, \eta_{\ell}(u_\ell^{\underline{k} })^2
		\le
		\eta_{\ell}(\mathcal{R}_{\ell}; u_\ell^{\underline{k} })^2
	\)
	and concludes the proof.
\end{proof}

Adapting the proof techniques from~\cite[Proof of Theorem~1]{bfmps2025}, we establish the following \emph{conditional} tail summability result.

\begin{lemma}[conditional tail summability of final iterates]\label{theorem:full_linear_convergence}
Suppose formally that $\tau < 0$ so that Algorithm~\ref{algorithm:unconditional_afem_gmres} does not terminate. Suppose that the estimator satisfies \eqref{axiom:stability}--\eqref{axiom:reliability} and that the employed preconditioner satisfies~\eqref{eq:equivalence_preconditioner_energy_norm}.
Let $0 < \theta \le 1$, $C_{\textup{mark}} \ge 1$, $u_0^0 \in \mathcal{X}_0$, and $\Calg > 0$ be arbitrary. 
Recall $\lalg^\star$ from Lemma~\ref{lem:estimator_equivalence}.
If there exists $\ell_0 \in \N$ such that the solver parameter \(\lalg > 0\) satisfies $0<\lalg < \lalg^\star \theta^{1/2}$ for all $\ell \geq \ell_0$, then Algorithm~\ref{algorithm:unconditional_afem_gmres} guarantees tail summability of the quasi-error ${\sf H}_\ell^{\underline{k}}$ of the final iterate, i.e., there exist a constant $\Ctail \geq 1$ such that
 \begin{equation}\label{eq:tail:summability}
	\sum_{\ell' = \ell +1}^{\underline{\ell}-1}{\sf H}_{\ell'}^{\underline{k}}
	\le
	\Ctail
	 \, {\sf H}_{\ell}^{\underline{k}}
	\quad \text{for all } 0 \leq \ell < \underline{\ell}.
 \end{equation}
\end{lemma}

\begin{proof}
We first show that \(({\sf H}_{\ell}^{\underline{k}})_{\ell=0}^{\underline{\ell}}\) satisfies assumption~\eqref{eq:add_prop_summability} in Lemma~\ref{lemma:summability}.
Suppose that \({\sf H}_{\ell}^{\underline{k}} = 0\), i.e., \(\eta_{\ell}(u_\ell^{\underline{k}})=0\) and \(\norm{\vec{s}_{\ell}^k}_{\vec{P}_{\ell}^{-1}}=0\).
Then, it follows immediately that \(u_{\ell}^{\underline{k}} = u_{\ell}^\star = u^\star\).
By Céa's lemma~\eqref{eq:cea} and nested iteration, we obtain $u_{\ell}^{\underline{k}} = u_{\ell}^\star = u^\star = u_{\ell'}^\star = u_{\ell'}^{\underline{k}}$ for all $\ell' > \ell$, and hence also \({\sf H}_{\ell'}^{\underline{k}} = 0\) for all \(\ell' > \ell\).
Therefore,
\begin{equation}\label{eq:zero_cond_quasi-error}
 \forall \ell < \ell' < \underline{\ell} : \Big({\sf H}_{\ell}^{\underline{k}} = 0 \; \Longrightarrow \; {\sf H}_{\ell'}^{\underline{k}} = 0 \Big).
\end{equation}

The case $\underline{\ell} < \infty$ directly follows from the proof of Lemma~\ref{lemma:summability} for a constant depending on $\{{\sf H}_{\ell}^{\underline{k}}\}_{\ell=0}^{\underline{\ell}}$.
	Let $\ell \in \mathbb{N}$ with $\ell_0 \leq \ell$ and $(\ell+1,\underline{k}) \in \mathcal{Q}$. 
	The stopping criterion~\eqref{eq:single:termination} yields
	\begin{equation*}
		\eta_{\ell}(u_{\ell}^{\underline{k}}) 
		\le
		\eta_{\ell}(u_{\ell}^{\underline{k}}) +
		\norm{ \vec{s}^{\underline{k}}_{\ell} }_{\boldsymbol{P}^{-1}_\ell}
		\le (1 + \lalg) \,
	\eta_{\ell}(u_{\ell}^{\underline{k}}).
	\end{equation*}
    Hence, estimator equivalence~\eqref{eq:estimator_equivalence} yields the equivalence
	\begin{equation}
	\label{eq:quasi-error-estimator}
	{\sf H}_\ell^{\underline{k}} \eqreff{eq:computable_quasierror}= \eta_{\ell}(u_{\ell}^{\underline{k}}) + \norm{ \vec{s}^{\underline{k}}_{\ell} }_{\boldsymbol{P}^{-1}_\ell} 
	\simeq
	\eta_{\ell}(u_{\ell}^{\underline{k}})
	\eqreff{eq:estimator_equivalence}\simeq
	\eta_{\ell}(u_{\ell}^{\star}).
	\end{equation}
  Lemma~\ref{lem:estimator_equivalence} shows that Dörfler marking~\eqref{eq:doerfler} implies Dörfler marking for the exact discrete solution $u_\ell^\star$ with parameter $\theta' =\theta_{\textup{mark}}$ defined in~\eqref{eq:thetamark_two}, where $\theta' > 0$ since $\lalg < \lalg^\star \theta^{1/2} \leq \lalg^\star$.
	Together with stability~\eqref{axiom:stability}, reduction~\eqref{axiom:reduction}, and $ q_\theta \coloneqq 1 - (1 - \qred^2) \theta'$, we obtain
    \begin{align}
        \begin{split}
        \label{eq:reduction_est}
        \eta_{\ell+1}(u_\ell^{\star})^2
        \quad &\stackrel{\mathclap{\eqref{axiom:stability}, \eqref{axiom:reduction}}}{\leq} \quad
        \eta_\ell(\TT_{\ell+1} \cap \TT_\ell; u_\ell^{\star})^2
        +
        \qred^2 \,
        \eta_\ell(\TT_\ell \setminus \TT_{\ell+1}; u_\ell^{\star})^2
		\\
        &= \,
        \eta_\ell(u_\ell^{\star})^2
        -
        (1 - \qred^2) \,
        \eta_\ell(\TT_\ell \setminus \TT_{\ell+1}; u_\ell^{\star})^2
        \\
        &\le \,
        \eta_\ell(u_\ell^{\star})^2
        -
        (1 - \qred^2) \,
        \eta_\ell(\mathcal{M}_\ell; u_{\ell}^{\star})^2
       \eqreff{eq:doerfler}{\leq} 
       q_\theta^2 \, \eta_\ell(u_\ell^{\star})^2.
        \end{split}
    \end{align}
    Applying stability~\eqref{axiom:stability} again yields 
    \begin{align}
		\begin{split}
    \label{eq:single:estimator-reduction}
        \eta_{\ell+1}(u_{\ell+1}^{\star})
        &\eqreff{axiom:stability}\leq
        \eta_{\ell+1}(u_\ell^{\star})
        +
        \Cstab \,
        \enorm{
            u_{\ell+1}^{\star} - u_\ell^{\star}
		} 
        \eqreff{eq:reduction_est}\leq
        q_\theta \,
        \eta_\ell(u_\ell^{\star})
        +
        \Cstab \,
        \enorm{
            u_{\ell+1}^{\star} - u_\ell^{\star}
		}
		\end{split}
	\end{align}
	for all $\ell_0 \leq \ell < \underline{\ell}$.
	Using quasi-orthogonality~\eqref{axiom:orthogonality} and reliability~\eqref{axiom:reliability}, we can estimate
	\begin{align}\label{eq:single:orthogonality}
		\begin{split}
			\sum_{\ell' = \ell}^{\ell+N} 
			\enorm{ u_{\ell'+1}^\star - u_{\ell'}^\star }^2
			\stackrel{\eqref{axiom:orthogonality}}\lesssim
			(N+1)^{1-\delta} \, \enorm{u^\star - u_\ell^\star}^2
			\stackrel{\eqref{axiom:reliability}}\lesssim
			(N+1)^{1-\delta} \, \eta_\ell(u_\ell^\star)^2.
		\end{split}
	\end{align}
	for all $0 \le \ell \le \ell+N < \underline{\ell}.$
	Hence, we can apply \cite[Lemma~9]{blp2026} for $\ell_0 \leq \ell$ and conclude 
	tail summability of $\{{\sf H}_\ell^{\underline{k}}\}_{\ell \geq \ell_0}$, i.e.,
	\begin{equation}\label{eq:single:summability-ell}
		\boxed{\sum_{\ell' 
		= \ell+1}^{\underline{\ell} - 1} {\sf H}_{\ell'}^{\underline{k}}
			\lesssim {\sf H}_\ell^{\underline{k}}
			\quad \text{for all } \ell_0 \le \ell < \underline{\ell}.}
	\end{equation}
	For a potentially larger constant depending also on $({\sf H}_{\ell}^{\underline{k}})_{\ell=0}^{\ell_0}$,~\eqref{eq:zero_cond_quasi-error} and Lemma~\ref{lemma:summability} yield tail summability of ${\sf H}_\ell^{\underline{k}}$ for all levels $0 \leq \ell < \underline{\ell}$, which concludes the proof.
\end{proof}

The following result guarantees that $\Calg$ and $\lalg$ are updated only finitely often.

\begin{lemma}[bounded number of updates of \(\boldsymbol{\Calg}\) and \(\boldsymbol{\lalg}\)]\label{lemma:bounded_updates}
Suppose formally that $\tau < 0$ so that Algorithm~\ref{algorithm:unconditional_afem_gmres} does not terminate. Suppose that the estimator satisfies \eqref{axiom:stability}--\eqref{axiom:reliability} and that the employed preconditioner satisfies~\eqref{eq:equivalence_preconditioner_energy_norm}.
Let $0 < \theta \le 1$, $C_{\textup{mark}} \ge 1$, $u_0^0 \in \mathcal{X}_0$, $\lalg > 0$, and $\Calg > 0$ be arbitrary. 
Then, the input parameters \(\Calg\) and \(\lalg\) in Algorithm~\ref{algorithm:unconditional_afem_gmres} are updated only finitely many times.
\end{lemma}

\begin{proof}
We argue by contradiction and assume that the parameters \(\Calg\) and \(\lalg\) are updated infinitely many times. 
Then, there exists a subsequence $\ell(n)$ such that 
\begin{equation}\label{eq:bad_subsequence}
	\sum_{j=0}^{\ell(n)-1} ({\sf H}_{j}^{\underline{k}})^{-1} > C_{\ell(n)} \, ({\sf H}_{\ell(n)}^{\underline{k}})^{-1} \quad \text{for all } n \in \N, 
\end{equation}
where $C_{\ell(n)} \rightarrow \infty$ for $n \to \infty$.
Since \(\lalg \to 0\) as \(\ell \to \infty\) even for the whole sequence, there exists an $\ell_0 \in \N$ such that $\lalg < \lalg^\star \theta^{1/2}$ for all $\ell \geq \ell_0$.
Consequently, the assumptions of Proposition~\ref{theorem:full_linear_convergence} are satisfied and yield tail summability of ${\sf H}_{\ell}^{\underline{k}}$ in the levels $\ell$.
Lemma~\ref{lemma:summability} shows that~\eqref{eq:tail:summability} is equivalent to inverse summability~\eqref{eq:inverse_summability}, i.e.,
\begin{equation*}
    \sum_{j=0}^{\ell-1} ({\sf H}_{j}^{\underline{k}})^{-1} \leq \Cinv \,  ({\sf H}_{\ell}^{\underline{k}})^{-1} \quad \text{for all } \ell \in \N_0
\end{equation*}
for some constant $\Cinv >0$.
This contradicts~\eqref{eq:bad_subsequence} and concludes the proof.
\end{proof}

As a consequence of Lemmas~\ref{theorem:full_linear_convergence} and~\ref{lemma:bounded_updates}, adapting the proof techniques from~\cite[Proof of Theorem~2]{bfmps2025} yields unconditional full R-linear convergence of Algorithm~\ref{algorithm:unconditional_afem_gmres}.

\begin{theorem}[unconditional full R-linear convergence of Algorithm~\ref{algorithm:unconditional_afem_gmres}]\label{theorem:unconditional_full_R_linear_convergence}
Suppose formally that $\tau < 0$ so that Algorithm~\ref{algorithm:unconditional_afem_gmres} does not terminate. Suppose that the estimator satisfies \eqref{axiom:stability}--\eqref{axiom:reliability} and that the employed preconditioner satisfies~\eqref{eq:equivalence_preconditioner_energy_norm}.
Let $0 < \theta \leq 1$, $\Cmark \geq 1$, $u_0^0 \in \XX_0$, $\lalg >0$, and $\Calg > 0$ be arbitrary. 
Then, Algorithm~\ref{algorithm:unconditional_afem_gmres} guarantees full R-linear convergence of the quasi-error ${\sf H}_{\ell}^k$ from~\eqref{eq:computable_quasierror}, i.e., there exist constants $0 < \qlin < 1$ and $\Clin > 0$ such that
\begin{equation}\label{eq:unconditional_full_R_linear_convergence}
{\sf H}_\ell^k \le C_{\textup{lin}} q_{\textup{lin}}^{|\ell,k| - |\ell',k'|} \, {\sf H}_{\ell'}^{k'}
\quad \text{for all }(\ell',k'),(\ell,k) \in \mathcal{Q} \text{ with } |\ell',k'| \le |\ell,k|.
\end{equation}
\end{theorem}

\begin{proof}
The proof is split into two steps.

\textbf{ Step~1 (tail summability with respect to $\boldsymbol{\ell}$).}
Lemma~\ref{lemma:bounded_updates} guarantees that \(\Calg\) and \(\lalg\) are updated only finitely often, i.e., there exists \(\ell_0 \in \N_0\) and \(C > 0\) such that 
\begin{equation}\label{eq:tail_summability}
    \sum_{j=0}^{\ell-1} ({\sf H}_{j}^{\underline{k}})^{-1} \leq  C \,  ({\sf H}_{\ell}^{\underline{k}})^{-1} \quad \text{for all } \ell_0 \leq \ell \in \N_0.
\end{equation}
For a potentially bigger constant \(C > 0\) depending also on $({\sf H}_{\ell}^{\underline{k}})_{\ell=0}^{\ell_0}$,~\eqref{eq:zero_cond_quasi-error} and Lemma~\ref{lemma:summability} yield tail summability of ${\sf H}_{\ell}^{\underline{k}}$ in $\ell$.

\textbf{Step~2 (tail summability with respect 
	to $\boldsymbol{\ell}$ and $\boldsymbol{k}$).}
	First, for $0 \le k < k' < \underline{k}[\ell]$, using that the 
	solver stopping criterion~\eqref{eq:single:termination} has not been met, 
	we have
	\begin{align*}
		{\sf H}_\ell^{k'}
		&\eqreff{eq:computable_quasierror}=
		 \eta_\ell(u_\ell^{k'}) + \norm{\vec{s}^{k'}_\ell }_{\boldsymbol{P}^{-1}_\ell} 
		\stackrel{\eqref{eq:single:termination}}
		\lesssim
		\norm{ \vec{s}^{k'}_\ell }_{\boldsymbol{P}^{-1}_\ell} 
		\stackrel{\eqref{eq:contr_gmres}}{\le}
		q_{\textup{alg}}^{k'-k} \, \norm{ \vec{s}^{k}_\ell }_{\boldsymbol{P}^{-1}_\ell} 
		\le
		q_{\textup{alg}}^{k'-k} \, {\sf H}_\ell^k.
	\end{align*}
	Second, stability~\eqref{axiom:stability}, equivalence~\eqref{eq:equivalence_residual_norm}, and contraction~\eqref{eq:contr_gmres} prove that
	\begin{align*}
			{\sf H}_\ell^{\underline{k}}
			\, & \eqreff*{eq:computable_quasierror}{=} \,
			\eta_\ell(u_\ell^{\underline{k}}) +
			\norm{ \vec{s}^{\underline{k}}_\ell }_{\boldsymbol{P}^{-1}_\ell}  
			\eqreff*{axiom:stability}\lesssim
			\eta_\ell(u_\ell^{\underline{k}-1})  +
			\enorm{ u_\ell^{\underline{k}} - u_\ell^{\underline{k}-1} } + 
			\norm{ \vec{s}^{\underline{k}}_\ell }_{\boldsymbol{P}^{-1}_\ell} \\
			\, &\eqreff*{eq:equivalence_residual_norm}\lesssim \,
			\eta_\ell(u_\ell^{\underline{k}-1})  +
			\norm{ \vec{s}^{\underline{k}}_\ell }_{\boldsymbol{P}^{-1}_\ell} + \norm{ \vec{s}^{\underline{k}-1}_\ell }_{\boldsymbol{P}^{-1}_\ell}
			\eqreff*{eq:contr_gmres}\lesssim
			\eta_\ell(u_\ell^{\underline{k}-1})  +
			\norm{ \vec{s}^{\underline{k}-1}_\ell }_{\boldsymbol{P}^{-1}_\ell}
			=
			{\sf H}_\ell^{\underline{k}-1} 
	\end{align*}
	for all $(\ell,\underline{k}) \in \mathcal{Q}.$
	Hence, we may conclude
	\begin{equation}\label{eq:single:contraction-Lambda}
		\boxed{{\sf H}_\ell^{k'}
			\lesssim
			q_{\textup{alg}}^{k'-k} \, {\sf H}_\ell^k
			\quad \text{for all } 0 \le k \le k' \le \underline{k}[\ell].}
	\end{equation}
	Nested iteration $u_{\ell+1}^0 = u_\ell^{\underline{k}}$, 
	reduction~\eqref{axiom:reduction}, and~\eqref{eq:single:orthogonality} with $N=0$, proves
		\begin{align*}
		\begin{split}
		{\sf H}_{\ell+1}^0 &= 
		\eta_{\ell+1}(u_{\ell+1}^{0})
		+
		\norm{ \vec{s}^{0}_{\ell+1} }_{\boldsymbol{P}^{-1}_{\ell+1}}  
		\eqreff{eq:equivalence_residual_norm}\lesssim
		\eta_{\ell+1}(u_{\ell}^{\underline{k}})
		+
		\enorm{ u_{\ell+1}^\star - u_\ell^{\underline{k}} }
		\\
		&
		\stackrel{\eqref{eq:reduction_est}}{\lesssim} 
		\eta_{\ell}(u_{\ell}^{\underline{k}})
		+
		\enorm{ u_{\ell+1}^\star - u_\ell^{\underline{k}} }
		\le {\sf H}_\ell^{\underline{k}}
		+ \enorm{ u_{\ell+1}^\star - u_\ell^\star }
			\eqreff{eq:single:orthogonality}{\lesssim} {\sf H}_\ell^{\underline{k}} + \eta_{\ell}(u_{\ell}^{\star}) \eqreff{eq:apost_quasierror}\simeq {\sf H}_\ell^{\underline{k}}
		\end{split}
	\end{align*}
	for all $(\ell,\underline{k}) \in \mathcal{Q}$.
	This leads to
	\begin{equation}\label{eq2:single:contraction-Lambda}
		\boxed{{\sf H}_{\ell+1}^0
			\lesssim {\sf H}_\ell^{\underline{k}}
			\quad \text{for all } (\ell,\underline{k}) \in \mathcal{Q}.}
	\end{equation}
	Overall, the geometric series proves tail summability
	\begin{align*}
		\sum_{\substack{(\ell',k') 
		\in \mathcal{Q} \\ |\ell',k'|> |\ell,k|}} {\sf H}_{\ell'}^{k'}
		& =
		\sum_{k'=k+1}^{\underline{k}[\ell]} {\sf H}_\ell^{k'}
		+ \sum_{\ell' 
		= \ell+1}^{\underline{\ell}} \sum_{k'=0}^{\underline{k}[\ell']} 
		{\sf H}_{\ell'}^{k'}
		\\[-2mm]&
		\stackrel{\mathclap{\eqref{eq:single:contraction-Lambda}}}\lesssim
		{\sf H}_\ell^k + \sum_{\ell' = \ell+1}^{\underline{\ell}} 
		{\sf H}_{\ell'}^0
		\stackrel{\eqref{eq2:single:contraction-Lambda}}\lesssim
		{\sf H}_\ell^k + \sum_{\ell' = \ell}^{\underline{\ell}-1} 
		{\sf H}_{\ell'}^{\underline{k}}
		\stackrel{\eqref{eq:single:summability-ell}}\lesssim
		{\sf H}_\ell^k + {\sf H}_\ell^{\underline{k}}
		\stackrel{\eqref{eq:single:contraction-Lambda}}\lesssim
		{\sf H}_\ell^k
		\quad \text{for all } (\ell,k) \in \mathcal{Q}.
	\end{align*}
	Finally, Lemma~\ref{lemma:summability} concludes the proof of full R-linear convergence~\eqref{eq:unconditional_full_R_linear_convergence}.
\end{proof}

\begin{corollary}[unconditional plain convergence of Algorithm~\ref{algorithm:unconditional_afem_gmres}]\label{corollary:convergence}
Under the assumptions of Theorem~\ref{theorem:unconditional_full_R_linear_convergence}, there holds convergence 
	\begin{equation}\label{eq:plain:convergence}
		\enorm{u^\star- u_\ell^k} + \enorm{u^{\star}-u_{\ell}^{\star}} + \enorm{u_\ell^{\star}-u_{\ell}^k} + \eta_{\ell}(u_{\ell}^{\star}) + \eta_{\ell}(u_{\ell}^k) \rightarrow 0 \quad \text{as }		
        |\ell,k| \rightarrow \infty.
	\end{equation}
\end{corollary}

\begin{proof} 
	From reliability~\eqref{axiom:reliability} and~\eqref{eq:apost_quasierror}, we deduce that
	\begin{align*}
		\enorm{u^\star- u_\ell^k} &+ \enorm{u^{\star}-u_{\ell}^{\star}} + \enorm{u_\ell^{\star}-u_{\ell}^k} + \eta_{\ell}(u_{\ell}^{\star}) \leq 2 \big(\enorm{u^{\star}-u_{\ell}^{\star}} + \enorm{u_\ell^{\star}-u_{\ell}^k}\big) + \eta_{\ell}(u_{\ell}^{\star}) \\
		& \eqreff{axiom:reliability}\leq (1+ 2 \Crel) \, \eta_{\ell}(u_{\ell}^{\star}) + 2 \, \enorm{u_\ell^{\star}-u_{\ell}^k} \eqreff{eq:apost_quasierror}\leq \max\{1+ 2 \Crel, 2\} \, C_{\textup{eq}} \, {\sf H}_\ell^k.
	\end{align*}
	Full R-linear convergence~\eqref{eq:unconditional_full_R_linear_convergence} for 
	$(\ell',k') = (0,0) \in \mathcal{Q}$ concludes the proof of~\eqref{eq:plain:convergence} via
    \begin{align*}
		\enorm{u^\star- u_\ell^k} + \enorm{u^{\star}-u_{\ell}^{\star}} &+ \enorm{u_\ell^{\star}-u_{\ell}^k} + \eta_{\ell}(u_{\ell}^{\star}) + \eta_{\ell}(u_{\ell}^k)\leq (1+\max\{1+ 2 \Crel, 2\} \, C_{\textup{eq}})  {\sf H}_\ell^k
    \\
	&\eqreff*{eq:unconditional_full_R_linear_convergence} \le (1+\max\{1+ 2 \Crel, 2\} \, C_{\textup{eq}}) \,
		C_{\textup{lin}} \, q_{\textup{lin}}^{|\ell,k|} \, {\sf H}_{0}^{0} \xrightarrow{|\ell,k| \rightarrow \infty} 0. \qedhere
	\end{align*}
\end{proof}


\section{Optimal complexity of Algorithm~\ref{algorithm:unconditional_afem_gmres}} \label{section:adaptive_algorithm}

First, we comment on the computational cost of Algorithm~\ref{algorithm:unconditional_afem_gmres}:
The cost of PGMRES for any step $k$ on a level $\ell$ is $O(\# \TT_{\ell})$ as discussed in Remark~\ref{label:remark_cost_PGMRES}.
Calculating the error estimator $\eta_{\ell}(u_{\ell}^{k})$ also has (up to quadrature) cost $O(\# \TT_{\ell})$.
The marking step can be implemented with cost $O(\# \TT_{\ell})$; see~\cite{stevenson2007} for $\Cmark = 2$ and~\cite{pp2020} for $\Cmark = 1$.
Finally, the cost of mesh-refinement by NVB is also $O(\# \TT_{\ell})$; see, e.g.,~\cite{stevenson2008, dgs2023}.
Due to the cumulative structure of Algorithm~\ref{algorithm:unconditional_afem_gmres}, the total computational cost to compute $u_{\ell}^{k}$ is therefore, up to a multiplicative constant, given by
\begin{equation}\label{eq:def_cost}
	\texttt{cost}(\ell,k) \coloneqq \hspace{-0.2cm} \sum_{
				\substack{
					(\ell',k') \in \mathcal{Q}
					\\
					|\ell',k'| \le |\ell,k|
				}
			}
			\hspace{-0.4cm}
			\#\TT_{\ell'}.
\end{equation}

A crucial consequence of full R-linear convergence is the following, which is also implicitly found in ~\cite{ghps2021}. 
If the rate of convergence $s>0$ is achievable with respect to the degrees of freedom, it is
also achievable with respect to the cumulative computational costs. Hence, full
R-linear convergence is the key to optimal complexity.
\begin{corollary}
	[rates = complexity {\cite[Corollary~1]{bfmps2025}}]
	\label{corollary:rates:complexity}
	Suppose full R-linear convergence~\eqref{eq:unconditional_full_R_linear_convergence}. Then, for
	any $s > 0$, it holds that
	\begin{equation}\label{eq:equiv-cost-rates}
		\medmuskip = -4mu
		M(s)
		\hspace{-0.1cm}
		\coloneqq
		\hspace{-0.2cm}
		\sup_{(\ell,k) \in \mathcal{Q}} (\#\TT_\ell)^s \,{\sf H}_\ell^{k}
		\le
		\hspace{-0.2cm}
		\sup_{(\ell,k) \in \mathcal{Q}}
		\Bigl(
			\hspace{-0.2cm}
			\sum_{
				\substack{
					(\ell',k') \in \mathcal{Q}
					\\
					|\ell',k'| \le |\ell,k|
				}
			}
			\hspace{-0.4cm}
			\#\TT_{\ell'}
		\Bigr)^s
		{\sf H}_\ell^{k}
		\le
		\frac{\Clin}{\bigl(1 \ - \ {\qlin}^{1/s}\bigr)^s} \, M(s).
	\end{equation}
	Moreover, there exists $s_0 > 0$ such that $M(s) < \infty$ for all $0 < s \le s_0.\qed$
\end{corollary}

To describe whether the
solution $u^\star$ can be approximated at rate $s>0$, we use the notion of
nonlinear approximation 
classes~\cite{bbdp02,bdd2004, stevenson2007, ckns2008, cfpp2014} 
\begin{equation*}
	\norm{u^\star}_{\A_s}
	\coloneqq
	\sup_{N \in \N_0}
	\Bigl(
		\bigl( N+1 \bigr)^s
		\min_{\TT_{\rm opt} \in \T_N } \eta_{\rm opt}(u^\star_{\rm opt})
	\Bigr),
\end{equation*}
where $\eta_{\rm opt}(u^\star_{\rm opt})$ is the estimator for
the (unavailable) exact Galerkin solution \(u_{\textup{opt}}^\star\) on
an optimal mesh $\TT_{\rm opt} \in \T_N$. 
The second main result of this work states that for sufficiently small parameters $\theta$ and $\lalg$ the adaptive algorithm
leads to optimal (in the sense of the nonlinear
approximation classes) decrease of the quasi-error with respect to 
overall computational cost.
The proof can be derived following steps analogous to~\cite[Proof of Theorem~8]{ghps2021}. 
For details, we refer to Appendix~\ref{section:appendixA}.

\begin{theorem}[optimal complexity]\label{th:optimal_complexity}
	Consider arbitrary $0 < \theta \le 1$,
	$0 < \lalg$, and arbitrary \(u_0^{0} \in \XX_0\).
	Suppose that the estimator satisfies
	\eqref{axiom:stability}--\eqref{axiom:discrete_reliability} and that the preconditioner satisfies~\eqref{eq:equivalence_preconditioner_energy_norm}.
	Recall $\lalg^\star$ and $\theta_{\textup{mark}}$ from Lemma~\ref{lem:estimator_equivalence}. Let $\theta^\star \coloneqq
			(1 + \Cstab^2 \, \Cdrel^2)^{-1}$.
	Suppose that $\theta$, $\lalg$ (as the input for Algorithm~\ref{algorithm:unconditional_afem_gmres}) are sufficiently small
	in the sense of
	\begin{equation}\label{eq:thetamark_two}
		0 < \lalg < \lalg^\star
		\quad
		\text{and}
		\quad
		0
		<
		\theta_{\textup{mark}}
		=
		\frac{
			(\theta^{1/2}
			+ \, \lalg / {\lalg^\star} )^{2}
			}{
			(1 - \lalg / {\lalg^\star} )^{2}
			}
		<
		\theta^\star
		<
		1, 
\end{equation}
Then, Algorithm~\ref{algorithm:unconditional_afem_gmres} guarantees, 
	for all $s > 0$, that
	\begin{equation}\label{eq:optimal_complexity}
		\copt \, \norm{u^\star}_{\A_s}
		\le
		\sup_{(\ell,k) \in \QQ}
		\Bigl(
			\sum_{
				\substack{
					(\ell', k') \in \QQ
					\\
					|\ell', k'| \le |\ell, k|
					}
				}
				\#\TT_{\ell'}
		\Bigr)^{s} \,
		{\sf H}_\ell^{k}
		\le
		\Copt \, \max\{ \norm{u^\star}_{\A_s}, \, {\sf H}_{0}^{0}\}.
	\end{equation}
	The constant \(\copt > 0\) depends only on \(\Cstab\), the use of
	NVB, and \(s\); while the constant $\Copt$ depends
	only on $\Cstab$, $\Cdrel$, $\Cmark$, $C_{\textup{mesh}}$, 
	\(q_{\textup{alg}}\),
	\(\lalg\), 
	$\theta$, $\# \TT_{0}$, and $s$.
\end{theorem}


\section{Numerical experiments} \label{section:numerical_experiments}

The performance of Algorithm~\ref{algorithm:unconditional_afem_gmres} is tested 
on the L-shaped domain \(\Omega := (-1, 1)^2 \setminus [0,1]\times[-1,0]\) with the 
diffusion coefficient \(\boldsymbol{K} = I_{2\times2}\),
 convection coefficient \(\mathbf{b} := (1, 25)^\top\), and 
 reaction coefficient \(c = 0\). 
The unknown analytical solution \(u^\star \in H^1_0(\Omega)\) satisfies 
\[
-\Delta u^\star + \begin{pmatrix} 1 \\ 25 \end{pmatrix} 
\cdot \nabla u^\star = 1 \quad \text{in} \quad \Omega.
\]
As seen in Figure~\ref{fig:Lshape}, this problem induces refinement in 
the reentrant corner of the domain as well as additionally on the boundary 
layer due to convection.

\begin{figure}
\resizebox{\textwidth}{!}{
		\includegraphics{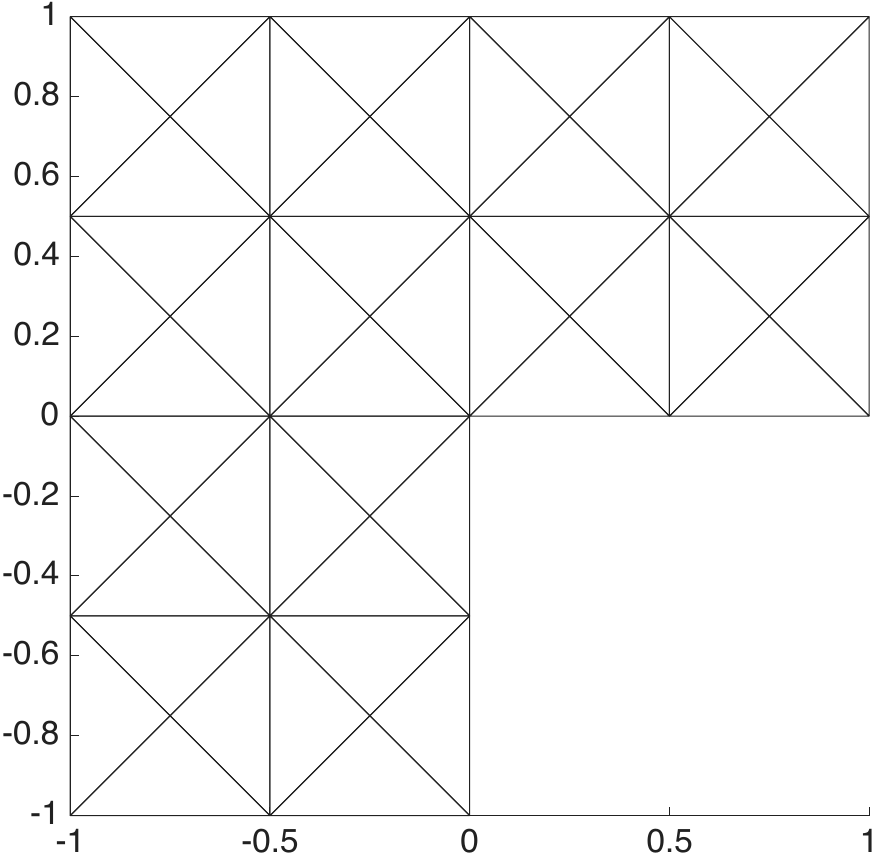} 
		\includegraphics{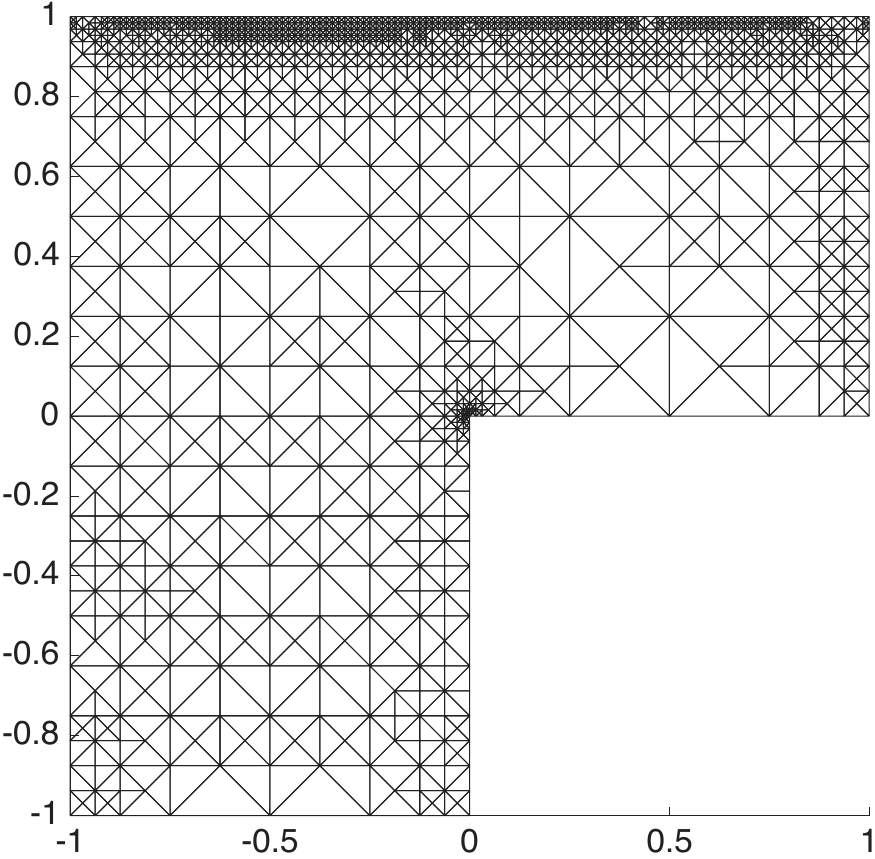}  
		\includegraphics{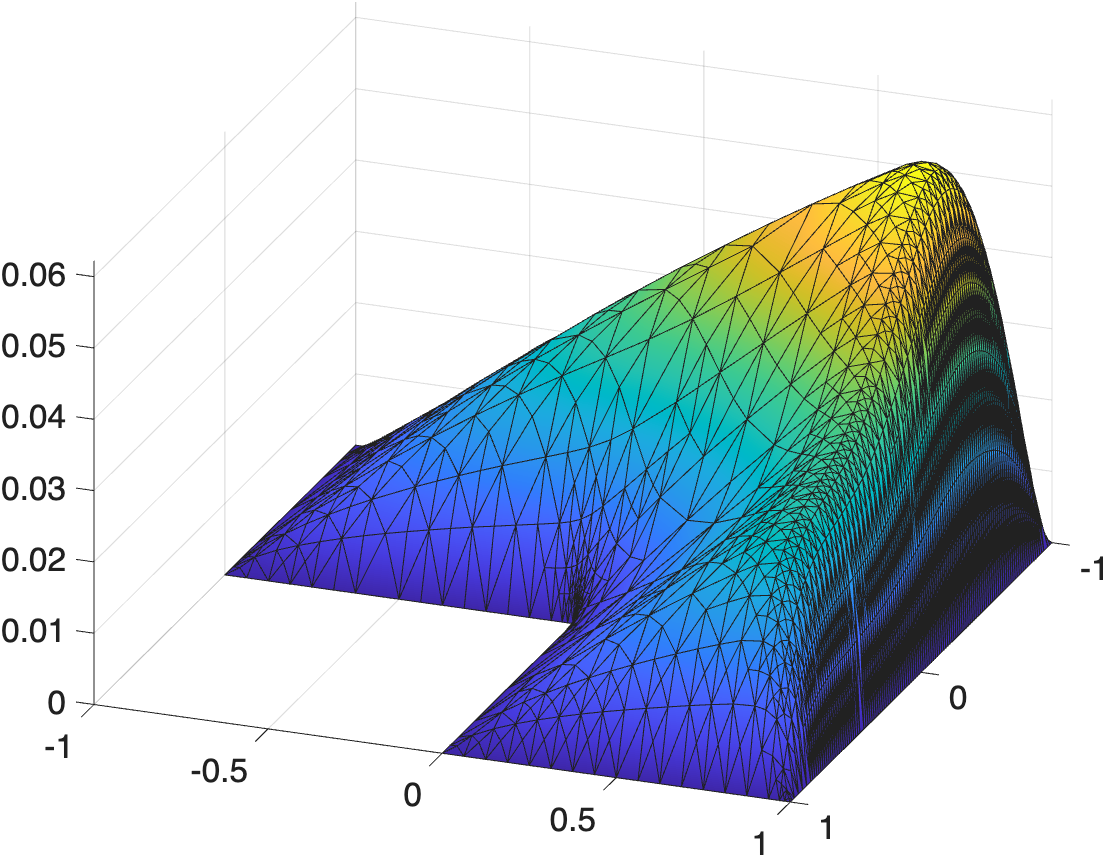}
	}
\caption{\label{fig:Lshape} 
	Initial mesh (left), with \(\# \mathcal{T}_0 = 48\) elements, used to 
	initialize 
	Algorithm~\ref{algorithm:unconditional_afem_gmres} together 
	with polynomial degree \(p = 2\) (leading to \(N_{0} = 81\) initial 
	degrees of freedom), bulk parameter \(\theta = 0.5\), and 
	maximum PGMRES steps
	\(k_{\max} = 5\). The resulting mesh (center) 
	after 10 refinements with \(
	\# \mathcal{T}_{10} = 3,\!396 \) elements (leading to \(N_{10} = 7,\!773\) 
	degrees of freedom) and corresponding approximation 
	\(u_{10}^{\underline{k}}\) (right). }
\end{figure}

\subsection{Optimality of the adaptive algorithm}

In the following, \(k_{\max}\) will denote the number of maximal iteration 
steps of PGMRES. Then, \(\underline{k} [\ell]\) refers to the overall iterations, 
with the understanding that for a given 
maximal iteration steps \(k_{\max}\) of PGMRES, 
\( R = \lfloor k/k_{\max} \rfloor \) denotes the number of 
PGMRES restart
steps.

First, we see in Figure~\ref{fig:opt_rates} that 
Algorithm~\ref{algorithm:unconditional_afem_gmres} indeed exhibits
optimal convergence rates with respect to number of degrees of 
freedom. Next, we are interested in optimal complexity 
of our adaptive algorithm, as stated in 
Theorem~\ref{th:optimal_complexity}. Indeed, we observe that
optimal convergence rates are achieved 
with respect to overall computational cost in the sense of~\eqref{eq:def_cost}, 
see Figure~\ref{fig:opt_cost}, as well as with respect to 
time, see Figure~\ref{fig:opt_comp}. 
In particular, optimal complexity 
holds also for the minimal approach of \(k_{\max} = 1\) PGMRES steps.

\begin{figure}
	\resizebox{\textwidth}{!}{
		\includegraphics{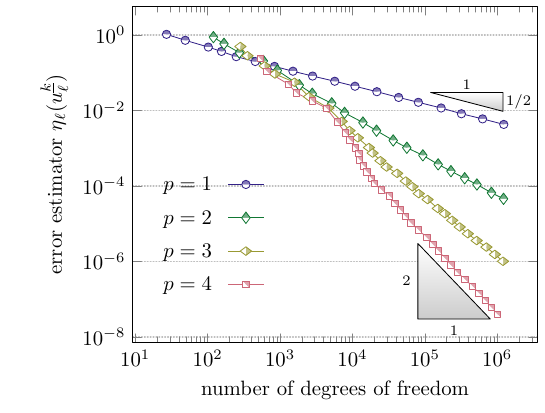}
		\includegraphics{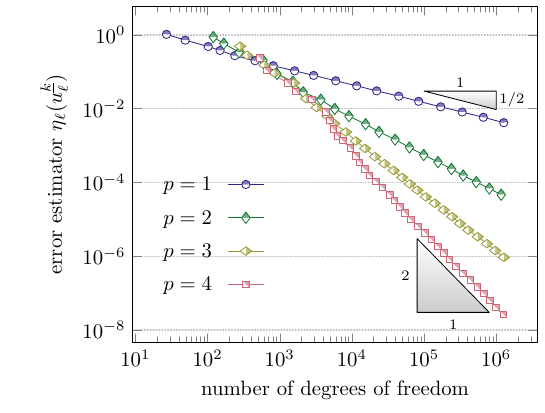}
	}
	\caption{\label{fig:opt_rates} 
	Rate-optimality of the Algorithm~\ref{algorithm:unconditional_afem_gmres} using the additive Schwarz preconditioner from~Section~\ref{sec:AS} 
	for GMRES on the linear system with 
	\(\theta = 0.5\) and \(k_{\max} = 5\) (left), \(k_{\max} = 1\) (right).}
\end{figure}

\begin{figure}
	\resizebox{\textwidth}{!}{
		\includegraphics{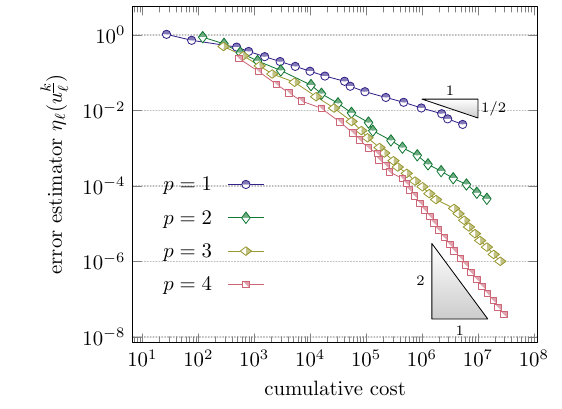}
		\includegraphics{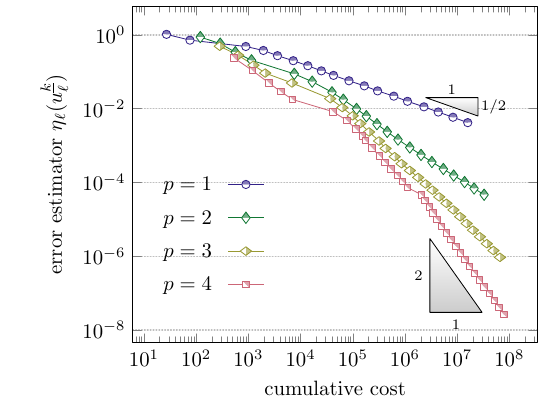} 
	}
	\caption{\label{fig:opt_cost} 
	Cost-optimality of the Algorithm~\ref{algorithm:unconditional_afem_gmres} 
	using the additive Schwarz preconditioner from~Section~\ref{sec:AS} for GMRES on 
	the linear system, 
	\(\theta = 0.5\) and \(k_{\max} = 5\) (left), \(k_{\max} = 1\) (right).}
\end{figure}

\begin{figure}
	\resizebox{\textwidth}{!}{
		\includegraphics{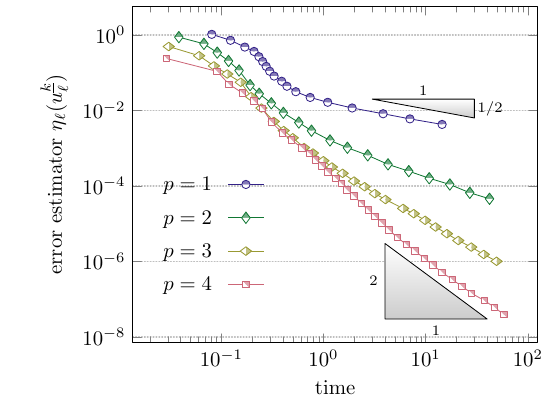}
		\includegraphics{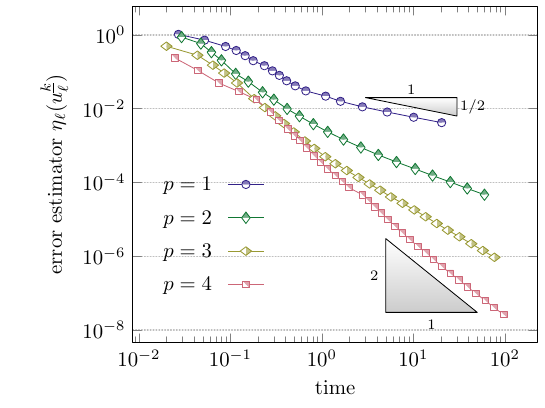}
	}
	\caption{\label{fig:opt_comp} 
	Optimal complexity of the Algorithm~\ref{algorithm:unconditional_afem_gmres} 
	using the
	additive Schwarz preconditioner from~Section~\ref{sec:AS} for 
	GMRES on the linear system, 
	\(\theta = 0.5\) and \(k_{\max} = 5\) (left), \(k_{\max} = 1\) (right).}
\end{figure}

\subsection{Adaptive stopping criteria}

First, we consider the performance of the new 
adaptive control of the parameters \(\Calg\) and \(\lalg\) in step (ii.a) in 
Algorithm~\ref{algorithm:unconditional_afem_gmres}. 
This plays a crucial role in proving unconditional convergence in Theorem~\ref{theorem:unconditional_full_R_linear_convergence}. 
Indeed, we know from Lemma~\ref{lemma:bounded_updates} that the updates can 
only occur finitely often. In practice, Figure~\ref{fig:parameters} 
illustrates that for several polynomial degrees, these parameters only 
need to be updated twice when initialized with 1. Thus, the resulting 
parameter choices appear to be in a standard range (neither too large, nor too
small) for a well-performing 
adaptive algorithm. 
In particular, the update of \(\Calg\) and \(\lalg\) seems to be independent of the choice of 
\(k_{\max}\).

\begin{figure}
\resizebox{\textwidth}{!}{
		\includegraphics{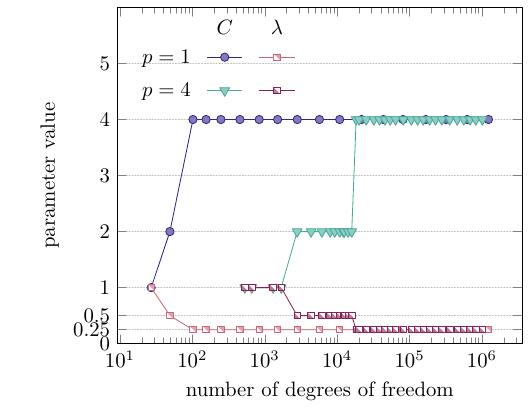} 
		\includegraphics{numerics/Stopping_Kmax5.pdf} 
	}
\caption{\label{fig:parameters} 
	Parameter adaptation in 
	Algorithm~\ref{algorithm:unconditional_afem_gmres}(ii.a) 
	with respect to the degrees of freedom for 
	\(\theta = 0.5\) and
	\(k_{\max} = 5\) (left), \(k_{\max} = 1\) (right). }
\end{figure}

Next, let us discuss the stopping of the algebraic solver. 
While the equibalancing of discretization and algebra 
errors is now ensured by 
the adaptive regulation of the parameter \(\lalg\), we test two different 
numbers $k_{\max}$ of maximal iteration steps for PGMRES. 
As seen in Figure~\ref{fig:restart},
there is a trade-off between a low \(k_{\max}\) and higher number of 
potential restarts and vice-versa. Nonetheless, it is noticeable that 
the total number of performed steps \(\underline{k}\) remains 
generally smaller for
\(k_{\max} = 5\) when compared to \(k_{\max} = 1\).

\begin{figure}
\resizebox{\textwidth}{!}{
		\includegraphics{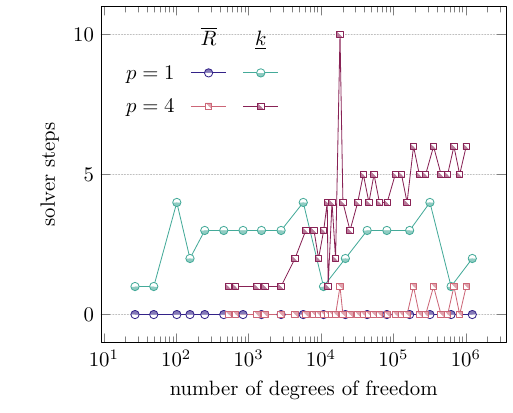} 
		\includegraphics{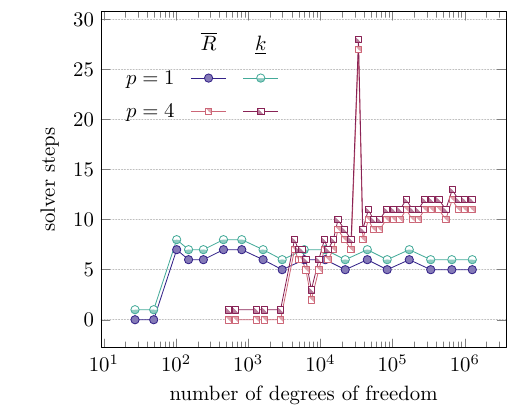} 
	}
\caption{\label{fig:restart} 
	Iteration steps and number of restarts $\overline{R} = \lfloor \underline{k}[\ell] / k_{\max}\rfloor$ for PGMRES 
	with additive Schwarz preconditioner from Section~\ref{sec:AS}
	when employing Algorithm~\ref{algorithm:unconditional_afem_gmres} 
	for \(\theta = 0.5\) and
	\(k_{\max} = 5\) (left) and \(k_{\max} = 1\) (right). }
\end{figure}

\subsection{Solver robustness}

To conduct some experiments tailored on the robustness of the 
iterative solver, we start by pre-computing a mesh hierarchy of 20 levels. 
We do so via adaptive refinement employing \(p=2\) and a direct solver. 
Next, we run optimally preconditioned GMRES via additive Schwarz from 
Section~\ref{sec:AS}, 
from a zero initial guess and stop the iterations when the 
algebraic relative residual drops below \(10^{-5}\). 

Figure~\ref{fig:PGMRES_res} shows overall \(p\)-robust contraction 
factors, 
as expected theoretically on the contraction bound of the preconditioned 
residuals \( \norm{\boldsymbol{r}^k_{\ell}}_{\boldsymbol{P}_\ell^{-1}}  \) from 
Proposition~\ref{lem:contraction_gmres}. Note that, since 
Proposition~\ref{lem:contraction_gmres} 
is ensured at each solver step, we indeed observe it numerically for 
different choices of \(k_{\max}\) of PGMRES, in particular also for 
\(k_{\max} =1\), i.e., we can restart at each step. 
Figure~\ref{fig:PGMRES_res} (left) illustrates how the number of restarts can 
make the contraction factors very oscillatory.
For \(\mod{k}{k_{\max}} = 0\), the Krylov space attains its maximal dimension \(k_{\max}\), so the contraction factor is minimal within the corresponding GMRES run.
After each restart, the contraction factor increases; see Figure~\ref{fig:PGMRES_res} (left).

Experimentally we observe in Figure~\ref{fig:PGMRES_en} (left) 
even better algebraic error contraction in the 
energy norm \( \enorm{ u^\star_{\ell}- u_{\ell}^k } \), 
though this is outside the scope of the theoretic framework. 
Finally, we observe in Figure~\ref{fig:PGMRES_en} the behavior for a 
non-restarted PGMRES. This is done here using \(k_{\max} =300\) which 
leads to a full Arnoldi basis and thus to better contraction as seen 
in Figure~\ref{fig:PGMRES_en} (left) and faster algebraic error decay as seen 
in Figure~\ref{fig:PGMRES_en} (right). 

\begin{figure}
\resizebox{\textwidth}{!}{
		\includegraphics{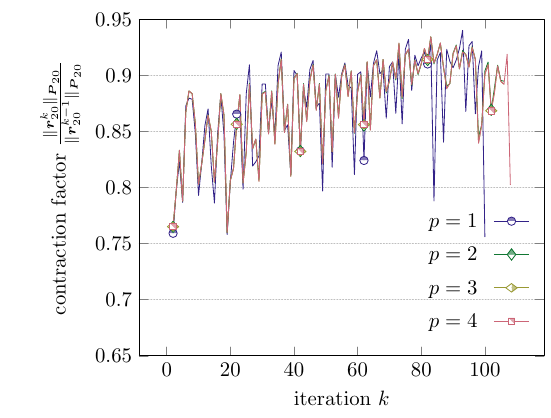} 
		\includegraphics{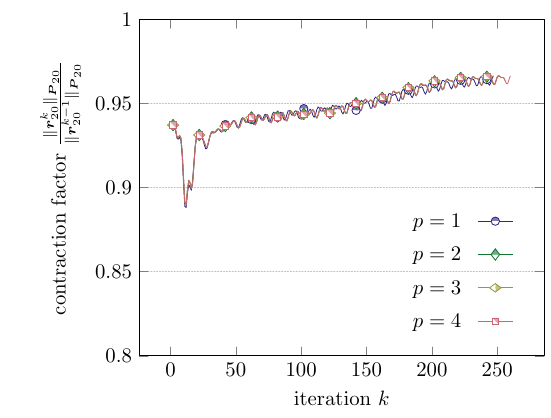} 
	}
\caption{\label{fig:PGMRES_res} 
	Contraction factors of the preconditioned residual for
	\(k_{\max} = 5\) (left) and \(k_{\max} = 1\) (right) for different 
	polynomial degrees on a fixed mesh $\TT_{20}$ (obtained by Algorithm~\ref{algorithm:unconditional_afem_gmres} for $p=2$) with $\# \TT_{20} = 99,\!085$ elements.}
\end{figure}

\begin{figure}
\resizebox{\textwidth}{!}{
		\includegraphics{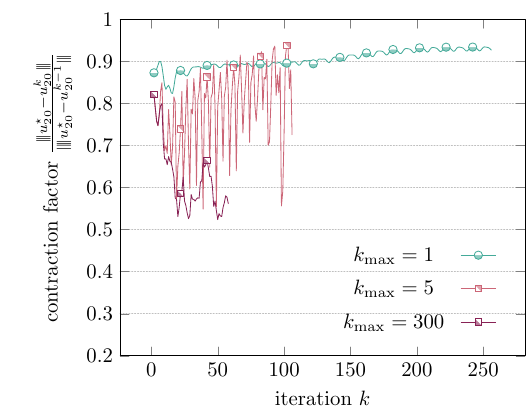} 
		\includegraphics{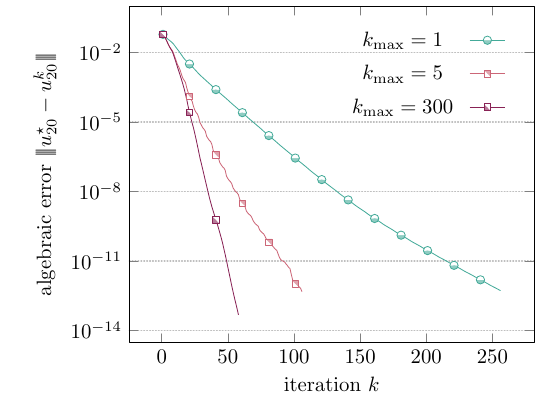} 
	}
\caption{\label{fig:PGMRES_en} 
	Contraction factors (left) and decrease (right) of the 
	algebraic error in energy norm for 
	\(p=3\) and different choices of
	\(k_{\max}\) on a fixed mesh $\TT_{20}$ (obtained by Algorithm~\ref{algorithm:unconditional_afem_gmres} for $p=2$) with $\# \TT_{20} = 99,\!085$ elements.}
\end{figure}

\renewcommand*{\bibfont}{\small} 
\printbibliography

\appendix


\section{Proof of Theorem~\ref{th:optimal_complexity}} \label{section:appendixA}

\begin{proof}[\textbf{Proof of Theorem~\ref{th:optimal_complexity}}]
	Thanks to Corollary~\ref{corollary:rates:complexity}, it suffices to show
	\begin{equation}\label{eq:proof:optimality}
		\norm{ u}_{\A_s}
		\lesssim
		\sup_{(\ell, k) \in \QQ} (\# \TT_{\ell})^s \, {\sf H}_{\ell}^{k}
		\lesssim
		\max \{\norm{ u}_{\A_s}, {\sf H}_{0}^{0}\}.
	\end{equation}

	\medskip
	\textbf{Step~1.} We first show the lower bound
	in~\eqref{eq:proof:optimality} for \(\underline{\ell} = \infty\). 
	In this case,
	Algorithm~\ref{algorithm:unconditional_afem_gmres} leads to \(\# \TT_{\ell} \to \infty\) 
	as
	\(\ell \to \infty\). For any \(N \in \N\), we follow the proof of
	\cite[Proposition~4.15]{cfpp2014} and choose the maximal index
	\(\ell' \in \N_0\) with \(\# \TT_{\ell'} - \# \TT_{0} \le N\). Since NVB
	refinement guarantees
	\(\# \TT_{\ell'+1} \lesssim \# \TT_{\ell'}\), we are led to
	\begin{equation*}
		N + 1 < \# \TT_{\ell'+1} - \# \TT_{0} + 1
		\le
		\# \TT_{\ell'+1}
		\lesssim
		\# \TT_{\ell'}.
	\end{equation*}
	Then, using \(\TT_{\ell'} \in \T\) and~\eqref{eq:apost_quasierror} we
	show for all \((\ell', k') \in \QQ\) that
	\begin{equation*}
		\min_{\TT_{\rm opt} \in \T_N} \eta_{\rm opt}(u_{\rm opt}^\star)
		\le
		\eta_{\ell'}(u_{\ell'}^\star)
		\eqreff{eq:apost_quasierror}\lesssim  {\sf H}_{\ell'}^{k'}.
	\end{equation*}
	A combination of the two previous estimates verifies
	\begin{equation*}
		(N + 1)^s \,
		\min_{\TT_{\rm opt} \in \T_N} \eta_{\rm opt}(u_{\rm opt}^\star)
		\lesssim
		(\# \TT_{\ell'})^s \, {\sf H}_{\ell'}^{k'}
		\lesssim
		\sup_{(\ell, k) \in \QQ} (\# \TT_{\ell})^s \, {\sf H}_{\ell}^{k}.
	\end{equation*}
	Taking the supremum over all \(N \in \N\), we prove the lower
	bound in~\eqref{eq:proof:optimality} for \(\underline{\ell} = \infty\).

	\medskip
	\textbf{Step~2.} We prove the lower bound in
	\eqref{eq:proof:optimality} for \(\underline{\ell} < \infty\).
    Corollary~\ref{corollary:convergence} yields that 
	\(\eta_{\underline{\ell}}(u_{\underline{\ell}}^\star) = 0\) and
	\(u_{\underline{\ell}}^\star = u^\star\). Since \(\norm{u^\star}_{\A_s} = 0\)
	for \(\underline{\ell} = 0\), we may assume \(\underline{\ell} > 0\). 
	Following Step~1, we let
	\(0 \le N < \# \TT_{\underline{\ell}} - \# \TT_{0} \) and 
	choose the maximal index
	\(0 \le \ell' < \underline{\ell}\) with 
	\(\# \TT_{\ell'} - \# \TT_{0} \le N\). With
	the estimates from Step~1, we arrive at
	\begin{equation*}
		\norm{ u^\star }_{\A_s}
		=
		\sup_{0 \le N < \# \TT_{\underline{\ell}} - \# \TT_{0}}
		\bigl(
			(N+1)^s
			\min_{\TT_{\rm opt} \in \T_N} \eta_{\rm opt}(u_{\mathrm{opt}}^\star)
		\bigr)
		\lesssim
		\sup_{(\ell, k) \in \QQ} (\# \TT_{\ell})^s \, {\sf H}_{\ell}^{k}.
	\end{equation*}
	This proves the lower bound
	in~\eqref{eq:proof:optimality} also for \(\underline{\ell} < \infty\).

	\medskip
	\textbf{Step~3.} We proceed to prove the upper bound
	in~\eqref{eq:proof:optimality}. To this end, we may suppose
	\(\norm{ u^\star }_{\A_s} < \infty\) since the result is trivial otherwise.
	Let $(\ell'+1, 0) \in \QQ$. 
	Then, \cite[Lemma~4.14]{cfpp2014} and $0 < \theta_{\textnormal{mark}} < \theta^{\star}$ guarantee
	the existence of a set \(\mathcal{R}_{\ell'} \subseteq \TT_{\ell'}\) 
	such that
	\begin{equation}\label{eq:Rell}
		\# \mathcal{R}_{\ell'}
		\lesssim
		\norm{ u^\star }_{\A_s}^{1/s} \, \eta_{\ell'}(u_{\ell'}^\star)^{-1/s}
		\quad
		\text{and}
		\quad
		\theta_{\textup{mark}} \, \eta_{\ell'}(u_{\ell'}^\star)^2
		\le
		\eta_{\ell'}(\mathcal{R}_{\ell'}; u_{\ell'}^\star)^2.
	\end{equation}
	Note that~\eqref{eq:equivalence_Doerfler} verifies that 
	\(\mathcal{R}_{\ell'}\) also
	satisfies the Dörfler marking criterion~\eqref{eq:doerfler} for $u_{\ell}^{\underline{k}}$ with 
	parameter \(\theta\), i.e.,
	\(
		\theta \, \eta_{\ell'}(u_{\ell'}^{\underline{k}})^2
		\le
		\eta_{\ell'}(\mathcal{R}_{\ell'}; u_{\ell'}^{\underline{k}})^2
	\).
	Therefore, it follows that
	\begin{equation}\label{eq:Doerfler_opt}
		\#\mathcal{M}_{\ell'}
		\le
		C_{\rm mark} \, \# \mathcal{R}_{\ell'}.
	\end{equation}
	Full linear convergence~\eqref{eq:unconditional_full_R_linear_convergence} and estimating as 
	in~\eqref{eq:quasi-error-estimator}, leads to
	\begin{align}\label{eq:Eta_estimator}
			{\sf H}_{\ell'+1}^{0}
			&\eqreff*{eq:unconditional_full_R_linear_convergence}
			\lesssim
			{\sf H}_{\ell'}^{\underline{k}}
			\eqreff{eq:quasi-error-estimator}\le
			\eta_{\ell'}(u_{\ell'}^{\underline{k}}).
	\end{align}
	Combining all estimates and using the estimator
	equivalence~\eqref{eq:estimator_equivalence}, we obtain
	\begin{equation}\label{eq:marked_Eta}
		\medmuskip = -1mu
		\# \mathcal{M}_{\ell'} \,
		\eqreff*{eq:Doerfler_opt}
		\lesssim \,
		\# \mathcal{R}_{\ell'} \,
		\eqreff*{eq:Rell}
		\lesssim \,
		\norm{ u^\star }_{\A_s}^{1/s} \, \eta_{\ell'}(u_{\ell'}^\star)^{-1/s}
		\eqreff*{eq:estimator_equivalence}
		\lesssim
		\norm{ u^\star }_{\A_s}^{1/s} \, 
		\eta_{\ell'}(u_{\ell'}^{\underline{k}})^{-1/s}
		\eqreff*{eq:Eta_estimator}
		\lesssim
		\norm{ u^\star }_{\A_s}^{1/s} \, 
		\bigl({\sf H}_{\ell'+1}^{0}\bigr)^{-1/s} .
	\end{equation}

	\medskip
	\textbf{Step 4.} Consider $(\ell, k) \in \QQ$, then full linear
	convergence~\eqref{eq:unconditional_full_R_linear_convergence} yields that
	\begin{equation}\label{eq:lin_cv_sum}
		\sum_{
			\substack{
				(\ell',k') \in \QQ
				\\
				|\ell',k'| \le |\ell,k|
			}
		}
		({\sf H}_{\ell'}^{k'})^{-1/s}
		\eqreff{eq:unconditional_full_R_linear_convergence}
		\lesssim
		({\sf H}_{\ell}^{k})^{-1/s}
		\sum_{
			\substack{
				(\ell',k') \in \QQ
				\\
				|\ell',k'| \le |\ell,k|
			}
		}
		(\qlin^{1/s})^{|\ell, k| - |\ell', k'|}
		\lesssim
		({\sf H}_{\ell}^{k})^{-1/s}.
	\end{equation}
	Recall that NVB refinement satisfies the mesh-closure estimate, i.e., there
	holds that
	\begin{equation}\label{eq:meshclosure}
		\# \TT_\ell - \# \TT_0
		\le
		C_{\textup{mesh}} \sum_{\ell' = 0}^{\ell-1} \# \mathcal{M}_{\ell'}
		\quad \text{for all }
		\ell \ge 0,
	\end{equation}
	where $C_{\textup{mesh}}  > 1$ depends only on $\TT_{0}$.
	Thus, for $(\ell, k) \in \QQ$ with $\ell > 0$, we have by the above
	formulas~\eqref{eq:marked_Eta}--\eqref{eq:meshclosure} that
	\begin{align*}
		\# \TT_{\ell} - \# \TT_{0} +1
		&\le
		2 \, (\# \TT_{\ell} - \# \TT_{0})
		\eqreff*{eq:meshclosure}
		\le
		2 \, C_{\textup{mesh}}  \sum_{\ell' = 0}^{\ell-1} \# \mathcal{M}_{\ell'}
		\eqreff*{eq:marked_Eta}
		\lesssim \
		\norm{u^\star}_{\A_s}^{1/s}
		\sum_{\ell' = 0}^{\ell-1} ({\sf H}_{\ell'+1}^{0})^{-1/s}
		\\
		&\lesssim
		\norm{u^\star}_{\A_s}^{1/s} \! \!
		\sum_{
			\substack{
				(\ell',k') \in \QQ
				\\
				|\ell',k'| \le |\ell,k|
			}
		}
		({\sf H}_{\ell'}^{k'})^{-1/s}
		\eqreff{eq:lin_cv_sum}
		\lesssim
		\norm{u^\star}_{\A_s}^{1/s} ({\sf H}_{\ell}^{k})^{-1/s}.
	\end{align*}
	Rearranging the terms, we obtain that
	\begin{subequations}\label{eq:Eta_approxclass}
		\begin{equation}
			(\# \TT_{\ell} - \# \TT_{0} + 1)^s \, {\sf H}_{\ell}^{k}
			\lesssim
			\norm{u^\star}_{\A_s}
			\quad \text{for all} \quad
			\ell > 0.
		\end{equation}
		Trivially, full linear convergence proves that
		\begin{equation}
			(\# \TT_{\ell} - \# \TT_{0} + 1)^s \, {\sf H}_{0}^{k}
			=
			{\sf H}_{0}^{k}
			\lesssim
			{\sf H}_{0}^{0}
			\quad \text{for} \quad
			\ell = 0.
		\end{equation}
	\end{subequations}
	Moreover, recall from~\cite[Lemma~22]{bhp2017} that, for all
	$\TT_H \in \T$ and all $\TT_h \in \T(\TT_H)$,
	\begin{equation}\label{eq:bhp-lemma22}
		\# \TT_h - \# \TT_H +1
		\le \# \TT_h
		\le  \# \TT_H \, (\# \TT_h - \# \TT_H +1).
	\end{equation}
	Overall, we thus have shown that
	\begin{equation*}
		(\# \TT_{\ell})^s {\sf H}_\ell^{k}
		\eqreff{eq:bhp-lemma22}
		\lesssim
		(\# \TT_{\ell} - \# \TT_{0} + 1)^s {\sf H}_\ell^{k}
		\eqreff{eq:Eta_approxclass}
		\lesssim
		\max \{\norm{u^\star}_{\A_s}, {\sf H}_{0}^{0}\}
		\
		\text{for all $(\ell, k) \in \QQ$.}
	\end{equation*}
	This concludes the proof of the upper bound in~\eqref{eq:proof:optimality}
	and hence that of~\eqref{eq:optimal_complexity}.
\end{proof}

\end{document}